\documentclass[pdflatex,sn-mathphys-num]{sn-jnl}

\usepackage{graphicx}%
\usepackage{multirow}%
\usepackage{amsmath,amssymb,amsfonts,mathtools}%
\usepackage{amsthm}%
\usepackage{mathrsfs}%
\usepackage[title]{appendix}%
\usepackage{textcomp}%
\usepackage{manyfoot}%
\usepackage{booktabs}%
\usepackage{algorithm}%
\usepackage{algorithmic}%
\numberwithin{algorithm}{section}
\usepackage{listings}%

\usepackage{comment}
\usepackage{svg}
\usepackage{lineno}
\usepackage[dvipsnames]{xcolor}
\usepackage{hyperref}
\usepackage{cleveref}
\usepackage{colortbl}
\numberwithin{equation}{section}

\newtheorem{theorem}{Theorem}[section]
\newtheorem{lemma}[theorem]{Lemma}
\newtheorem{corollary}[theorem]{Corollary}
\newtheorem{property}[theorem]{Property}

\newtheorem{remark}[theorem]{Remark}%

\newtheorem{definition}[theorem]{Definition}%

\crefname{lemma}{Lemma}{Lemmata}
\crefname{theorem}{Theorem}{Theorems}
\crefname{corollary}{Corollary}{Corollaries}
\crefname{remark}{Remark}{Remarks}
\crefname{property}{Property}{Properties}
\crefname{figure}{Fig.}{Figs.}

\newcommand{\Z}{\mathbb Z}

\newcommand{\R}{\mathbb R}

\newcount\Comments  
\newcommand{\kibitz}[2]{\ifnum\Comments=1\textcolor{#1}{#2}\fi}

\makeatletter
\newcommand{\cline}[1]{%
  \@cline#1\@nil
}

\def\@cline#1-#2\@nil{%
  \omit
  \@multicnt#1%
  \advance\@multispan\m@ne
  \ifnum\@multicnt=\@ne\@firstofone{&\omit}\fi
  \@multicnt#2%
  \advance\@multicnt -#1%
  \advance\@multicnt \@ne
  \leaders\hrule\@height\arrayrulewidth\hfill
  \hskip-\tabcolsep
  \cleaders\hbox{\vrule \@height\arrayrulewidth \@width\arrayrulewidth}\hfill
  \hskip\tabcolsep
  \advance\@multispan\@ne
  \global\advance\@multispan\@multicnt
  \ignorespaces
}
\makeatother

\begin{document}
\title{Shortest Geodesic Loops, Sectional Curvature, and Injectivity Radius of the Stiefel Manifold}


\author*[1]{\fnm{Jakob} \sur{Stoye} \email{jakob.stoye@tu-braunschweig.de}}
\author[2]{\fnm{Simon} \sur{Mataigne} \email{simon.mataigne@uclouvain.be}}
\author[2]{\fnm{P.-A.} \sur{Absil} \email{pa.absil@uclouvain.be}}
\author[3]{\fnm{Ralf} \sur{Zimmermann} \email{zimmermann@imada.sdu.dk}}

\affil*[1]{\orgdiv{Institute for Numerical Analysis}, \orgname{TU Braunschweig}, \orgaddress{\city{Braunschweig}, \country{Germany}}}
\affil[2]{\orgdiv{ICTEAM Institute}, \orgname{UCLouvain}, \orgaddress{\city{Louvain-la-Neuve}, \country{Belgium}}}
\affil[3]{\orgdiv{Department of Mathematics and Computer Science}, \orgname{University of Southern Denmark}, \orgaddress{\city{Odense}, \country{Denmark}}}
\abstract{
    We determine the length of the shortest nontrivial geodesic loops on the Stiefel manifold endowed with any member of the one-parameter family of Riemannian metrics introduced by Hüper et al. (2021). This family includes, in particular, the canonical and Euclidean metrics. By combining existing and new bounds on the sectional curvature, we determine the exact value of the injectivity radius of the Stiefel manifold under a wide range of members of the metric family. 
}
\keywords{Stiefel manifold, geodesic loops, injectivity radius, sectional curvature.}

\pacs[MSC Classification]{15B10, 15B57, 15A16, 22E70, 53C30, 53C80.}
\maketitle
{\scriptsize
\bmhead{Acknowledgements} Simon Mataigne is a Research Fellow of the Fonds de la Recherche Scientifique - FNRS. This work was supported by the Fonds de la Recherche Scientifique - FNRS under Grant no T.0001.23.}

\section{Introduction}

The Stiefel manifold is one of the most classical matrix manifolds~\cite{EdelmanAriasSmith:1999,boumal2023,Absilbook:2008}. It appears in numerous applications ranging from optimization~\cite{ChenMaManChoZhang:2024,sato2021} over numerical methods for differential equations~\cite{BennerGugercinWillcox2015,Celledoni_2020,HueperRollingStiefel2008,Zim21} to applications in statistics and data science~\cite{Chakraborty2019,Fletcher:2020,Turaga_2008}.
In particular, the geodesic endpoint problem on the Stiefel manifold has attracted a growing interest; see, e.g.,~\cite{sutti2023shootingmethodscomputinggeodesics,mataignezimmermann25}. 

Geodesics are manifold curves with zero covariant acceleration, generalizing straight lines in Euclidean spaces. The notion of \emph{covariant acceleration} depends on the selected Riemannian metric and therefore, so do the geodesics. The geodesic endpoint problem is the problem of finding a geodesic that connects two given points. On complete Riemannian manifolds, such a geodesic always exists but is not necessarily unique. Among all solutions to the geodesic endpoint problem, a \emph{minimal geodesic}, i.e., of shortest length, is especially desired since the Riemannian distance between the two points is defined by its length. Ensuring that a geodesic is minimal is therefore a crucial problem of Riemannian geometry for optimization on manifolds~\cite{Zhang22,bergmann2022manopt,Chattopadhyay16,Selvan15,hosseini2017}, statistics on manifolds~\cite{afsari2011,arnaudon2013,Pennec2006,guigui2023}, and data fitting~\cite{absil2016beziersurfaces,gousenbourger2019}.

The (\emph{global}) injectivity radius of a complete Riemannian manifold $\mathcal{M}$ is the largest $r\geq 0$ such that every geodesic of length less than or equal to $r$ is a shortest curve between its endpoints. The injectivity radius thus provides a sufficient condition on the length of a geodesic for its minimality to hold. Equivalent definitions can be found in Riemannian geometry textbooks such as~\cite{docarmo,petersen2016riemannian}. The injectivity radius is known in closed form for very few manifolds, such as the $n$-sphere or the Grassmann manifold of $p$-dimensional subspaces~\cite{Wong1967,Wong1968}. Lately, the injectivity radius of the Stiefel manifold under the Euclidean metric was added to that list~\cite{zimmermannstoye_eucl_inj:2024}.

A main tool in determining a manifold’s injectivity radius is due to Klingenberg~\cite{klingenberg1982}. If $K>0$ is a global upper bound on the sectional curvatures of the compact Riemannian manifold $\mathcal{M}$, then the injectivity radius is given by the length of a shortest (nontrivial) geodesic loop, if the latter is smaller than $\frac{\pi}{\sqrt{K}}$, and it is bounded from below by $\frac{\pi}{\sqrt{K}}$ otherwise. Klingenberg's statement requires thus the knowledge of i) a global upper bound on the sectional curvature of the manifold and ii) the length of the shortest geodesic loops.

\paragraph{Contributions.} In this work, we give the length of the shortest geodesic loops of the Stiefel manifold under every member of the one-parameter family of metrics that was introduced by H\"uper et al.~\cite{HueperMarkinaLeite2020}. 
This family of metrics is parameterized by the parameter $\beta > 0$ and the most prominent metrics on the Stiefel manifold are members of this metric family. Indeed, $\beta = \frac12$ gives the canonical metric and $\beta = 1$ gives the Euclidean metric. For both metrics, the length of the shortest geodesic loops of the Stiefel manifold have been determined; see~\cite{absilmataigne2024ultimate,Rentmeesters2013} for the canonical metric and~\cite{zimmermannstoye_eucl_inj:2024} for the Euclidean metric. In both cases, the shortest geodesic loops have a length of $2\pi$.
For $\beta\in(0,2]$, we use the essential observation of~\cite{zimmermannstoye_eucl_inj:2024} that the geodesics have constant-norm time derivatives, which makes them curves of constant Frenet curvatures, when considered as curves in the ambient Euclidean space $\R^N$. We first derive bounds on the Euclidean length of the curves and translate them in a second step to the $\beta$-metrics. For $\beta > 2$, we focus on the property that geodesic loops start and end at the same point to provide a matrix-analytic proof of the length of the shortest geodesic loops. In total, we prove that the length $\ell_\beta$ of the shortest geodesic loops on the Stiefel manifold under the family of $\beta$-metrics is given by $\ell_\beta = \min\{\sqrt{2\beta}, 1\}2\pi$.

A second contribution is that we provide the sharp upper bound $K_\beta$ on the sectional curvatures of the Stiefel manifold under the $\beta$-metrics for $\beta\in(\frac23,\frac{1}{\sqrt{2}})$. This closes a gap in the existing sharp upper bounds on the sectional curvature (see~\cite{zimmermannstoye_curvature:2024}), as there are now bounds on the sectional curvatures for $\beta\in(0,1]$. 

Finally, with the upper bound $K_\beta$ and the length $\ell_\beta$ of the shortest geodesic loops, we utilize Klingenberg's theorem to determine the exact value of the injectivity radius of the Stiefel manifold under the $\beta$-metrics for $\beta\in(0,\frac13]\cup[\frac23,1]$. For the $\beta$-intervals $(0,\frac13]$ and $[\frac23,1]$, the injectivity radius is given by $\sqrt{2\beta}\pi$ and $\pi$, respectively. For $\beta\in(\frac13,\frac23)$, the injectivity radius is bounded from above and below such that its value can differ at most $3\%$ from the conjectured injectivity radius of~\cite{absilmataigne2024ultimate}.

\paragraph{Organization of the paper.} In \Cref{sec:preliminaries}, we introduce essential notions of differential geometry and matrix manifolds. In \Cref{sec:betalength<=2} and \Cref{sec:closedgeod_beta>2}, we determine the length of the shortest geodesic loops for the cases $\beta \leq 2$ and $\beta > 2$, respectively. Between these sections, \Cref{sec:exp_inverse_and_ig} develops several properties of the matrix exponential inverse, needed in \Cref{sec:closedgeod_beta>2}. Subsequently, in \Cref{sec:bounds_sectcurvature}, we close the existing gap in the upper bound of the sectional curvature of the Stiefel manifold for $\beta \in (\frac{2}{3}, \frac{1}{\sqrt{2}})$. Finally, in \Cref{sec:injectivity}, we discuss the implications of the preceding sections for the injectivity radius of the Stiefel manifold.

\section{Preliminaries on the Stiefel manifold}\label{sec:preliminaries}
\subsection{The Stiefel manifold}
We begin by introducing the Stiefel manifold and outline basic concepts of differentiable manifolds such as geodesics, curvature, and the injectivity radius. Related references are~\cite{Absilbook:2008, gallier2011geometric,ZimmermannHueper2022}. The Stiefel manifold $\mathrm{St}(n,p)$ is the compact matrix manifold of orthogonal $p$-frames in $\mathbb{R}^n$, i.e., 
\begin{equation*}
    \mathrm{St}(n,p)\coloneq \{ U \in\mathbb{R}^{n\times p}\ | \  U^\top U = I_p\}.
\end{equation*}
It is an $np-\frac{p(p+1)}{2}$-dimensional embedded submanifold of $\mathbb{R}^{np}\cong\mathbb{R}^{n\times p}$. The Stiefel manifold is extensively studied in the literature; see, e.g.~\cite{EdelmanAriasSmith:1999,Absilbook:2008,Zim21}. A special point of the Stiefel manifold that will be used frequently is $I_{n\times p}\coloneq [I_p\ 0_{p\times (n-p)}]^\top$. For any point $U\in \mathrm{St}(n,p)$, the tangent space to $\mathrm{St}(n,p)$ at the point $U$ is given by
\begin{equation*}
    T_U\mathrm{St}(n, p) = \{\Delta\in\mathbb{R}^{n\times p}\ | \ U^\top \Delta\in\mathrm{Skew}(p)\}\subset \mathbb{R}^{n\times p}.
\end{equation*}
Every tangent vector $\Delta\in T_U \mathrm{St}(n,p)$ may be written as 
\begin{equation*}
    \Delta = UA + U_\perp B,\text{ with } A\in\mathrm{Skew}(p),\ B\in\mathbb{R}^{(n-p)\times p}\text{ and } [U\ U_\perp]\in\mathrm{O}(n).
\end{equation*}
This decomposition is obtained by taking $A = U^\top \Delta$ and $U_\perp B = (I-UU^\top)\Delta$. Letting $Q\coloneq [U \ U_\perp]$, we can write $\Delta = Q\left[\begin{smallmatrix}
    A\\
    B
\end{smallmatrix}\right]$. Notice that if $(I-UU^\top)\Delta$ is not of rank $n-p$, then $B$ has not a full row-rank.
 
On the Stiefel manifold, there exists a one-parameter family of metrics that was introduced  in
\cite{HueperMarkinaLeite2020} and studied in~\cite{nguyen2022curvature, nguyen2021closedform, ZimmermannHueper2022, absilmataigne2024ultimate, nguyen25}. Given the metric parameter $\beta \in (0,\infty)$, and any two vectors $\Delta_1,\Delta_2 \in T_U\mathrm{St}(n,p)$, $\Delta_i=Q\left[\begin{smallmatrix}
    A_i\\
    B_i
\end{smallmatrix}\right]$ for $i\in\{1,2\}$, the $\beta$-metric $\langle\cdot,\cdot\rangle_{\beta,U}:T_U\mathrm{St}(n,p)^2\rightarrow \mathbb{R} $ and its induced norm $\|\cdot \|_{\beta,U}:T_U\mathrm{St}(n,p)\rightarrow \mathbb{R}$ are defined by
\begin{equation*}
    \langle\Delta_1,\Delta_2\rangle_{\beta,U} = \beta \mathrm{tr}(A_1^\top A_2) + \mathrm{tr}(B_1^\top B_2),\quad \text{and}\quad \|\Delta_1\|_{\beta,U}= \sqrt{\langle\Delta_1,\Delta_1\rangle_{\beta,U}}.
\end{equation*}
For brevity, we drop the $U$-subscript and write $\langle\cdot,\cdot\rangle_{\beta}$ and $\|\cdot \|_{\beta}$. The Stiefel manifold endowed with the $\beta$-metric is denoted by $\mathrm{St}_\beta(n,p)$. 
To highlight the Euclidean case, we write occasionally 
$\langle\cdot,\cdot\rangle_{\mathrm{E}}$ and $\|\cdot \|_{\mathrm{E}}$ instead of $\langle\cdot,\cdot\rangle_{1}$ and $\|\cdot \|_{1}$.
The \emph{$\beta$-length} of a continuously differentiable curve $\gamma:[0,1]\rightarrow \mathrm{St}_\beta(n,p)$ is:
\begin{equation*}
    L_\beta(\gamma)\coloneq \int_0^1 \|\dot{\gamma}(t)\|_\beta \mathrm{d}t,
\end{equation*}
where $\dot{\gamma}$ denotes the time derivative of $\gamma$. W.l.o.g., in this paper, we always consider curves to be defined on the time interval $[0,1]$, unless explicitly stated otherwise. 
\subsection{Geodesics and the injectivity radius}\label{sec:geod_and_injrad_intro}
Geodesics are the generalization of straight lines in flat spaces to smooth manifolds, and candidates for length-minimizing curves, see, e.g.,~\cite[Definition~5.19]{kuhnel:2015}. 
Geodesics are uniquely defined by a starting point and a starting velocity; on the Stiefel manifold by $U\in \mathrm{St}(n,p)$, $\Delta= Q\left[\begin{smallmatrix}
    A\\
    B
\end{smallmatrix}\right] \in T_U\mathrm{St}(n,p)$, where $Q\coloneq [U \ U_\perp]\in\mathrm{O}(n)$.
By~\cite[eq. (11)]{ZimmermannHueper2022}, the associated geodesic is
\begin{equation*}
    \gamma_\beta(t) = Q \exp_\mathrm{m}\left(t\begin{bmatrix}
        2\beta A&-B^\top\\
        B&0
    \end{bmatrix}\right) I_{n\times p} \exp_{\mathrm{m}}(t(1-2\beta)A).
\end{equation*}
The Riemannian exponential is $\mathrm{Exp}_{\beta,U}:T_U\mathrm{St}(n,p)\rightarrow \mathrm{St}(n,p), \, \mathrm{Exp}_{\beta,U}(\Delta) = \gamma_\beta(1)$, and is well-defined, since $\mathrm{St}_\beta(n,p)$ is complete. An important theoretical tool for studying the Riemannian exponential and its invertibility is the \emph{injectivity radius}, see, e.g.,~\cite[Chap.~III, Def.~4.12]{sakai1996riemannian}.

\begin{definition}[Injectivity Radius]
    Let $\mathcal{M}$ be a complete Riemannian manifold, $q\in\mathcal{M}$ be a point on the manifold and $\mathrm{Exp}_q$ denote the Riemannian exponential at $q$.
    We define the injectivity radius $\mathrm{inj}_q(\mathcal{M})$ at $q$ as 
    \begin{align*}
        \sup\{r> 0;\;\mathrm{Exp}_q|_{B_r(q)}\text{ is a diffeomorphism}\}.
    \end{align*}
    The infimum of $\mathrm{inj}_q(\mathcal{M})$ over all $q\in\mathcal{M}$ is called the injectivity radius of $\mathcal{M}$ and is denoted by $\mathrm{inj}(\mathcal{M}) = \inf_{q\in\mathcal{M}}\mathrm{inj}_q(\mathcal{M})$.
\end{definition}

The Stiefel manifold $\mathrm{St}_\beta(n,p)$ is a homogeneous space for all $\beta>0$~\cite[Sect.~3.1~and~3.2]{absilmataigne2024ultimate}. Therefore, the geodesic starting at $QI_{n\times p}\in\mathrm{St}_\beta(n,p)$ with initial velocity $Q\left[\begin{smallmatrix}
    A\\
    B
\end{smallmatrix}\right]$ is the geodesic starting $I_{n\times p}$ with initial velocity $\left[\begin{smallmatrix}
    A\\
    B
\end{smallmatrix}\right]$, left multiplied by $Q$.
In consequence, the injectivity radius is independent of the starting point and, in particular, it follows that $\mathrm{inj}(\mathrm{St}_\beta(n,p)) = \mathrm{inj}_{I_{n\times p}}(\mathrm{St}_\beta(n,p))$.
 
A result due to Klingenberg~\cite{klingenberg1982} (or see~\cite[Chap.~III, Prop.~4.13]{sakai1996riemannian}) states that
\begin{align}\label{eq:inj=min(l/2,conj)}
    \mathrm{inj}(\mathrm{St}_\beta(n,p)) = \min\left\{\frac{\ell_\beta}{2}, \mathrm{conj}_{I_{n\times p}}(\mathrm{St}_\beta(n,p))\right\},
\end{align}
where $\ell_\beta$ is the length of a shortest \emph{geodesic loop} starting at $I_{n\times p}$ and $\mathrm{conj}_{I_{n\times p}}(\mathrm{St}_\beta(n,p))$ is the so-called conjugate radius (see~\cite[Chap.~III, Prop.~4.13]{sakai1996riemannian}). A \emph{geodesic loop} is a geodesic that begins and ends at the same point. Moreover, if the sectional curvature of $\mathrm{St}_\beta(n,p)$ is globally bounded by $K_\beta$, it follows that 
\begin{align}\label{eq:conjbound}
\mathrm{conj}_{I_{n\times p}}(\mathrm{St}_\beta(n,p))\geq \frac{\pi}{\sqrt{K_\beta}},
\end{align}
see, e.g.,~\cite[Chap.~III, Prop.~3.6]{sakai1996riemannian}.

Upper bounds on the conjugate radius are given in~\cite[Theorem 5.1]{absilmataigne2024ultimate} by finding explicit conjugate points, while lower bounds are given in~\cite{zimmermannstoye_curvature:2024} by determining the maximum sectional curvature $K_\beta$.
The shortest geodesic loops under the canonical and under the Euclidean metric have a length of $2\pi$, see 
\cite[Section 5]{Rentmeesters2013},~\cite[Theorem 6.1]{absilmataigne2024ultimate} (canonical case), and ~\cite{zimmermannstoye_eucl_inj:2024} (Euclidean case).
However, for $\beta\notin \{\frac{1}{2}, 1\}$, the length of the shortest geodesic loops is unknown but bounded from above by~\cite[Eq.~(6.1)]{absilmataigne2024ultimate}
\begin{equation}\label{eq:loop_upper_bound}
    \ell_\beta \leq \min\{\sqrt{2\beta}, 1\}2\pi.
\end{equation}
In this paper, we show that the upper bound from~\eqref{eq:loop_upper_bound} is an equality for the complete range $\beta\in(0,\infty)$. Together with the lower bounds on the conjugate radius (from~\cite{zimmermannstoye_eucl_inj:2024} and \Cref{sec:bounds_sectcurvature}), this finding allows us to determine the exact value of the injectivity radius of $\mathrm{St}_\beta(n,p)$ for $\beta\in(0,\frac13]\cup[\frac23,1]$ as well as defining bounds for the injectivity radius for all other values of $\beta$.
This is the content of \cref{thm:injectivity_radius}.

\section{The \texorpdfstring{$\beta$}{beta}-length of geodesic loops for \texorpdfstring{$\beta\in (0,2]$}{beta in (0,2]}}\label{sec:betalength<=2}

When $p=1$, the Stiefel manifold reduces to the unit sphere as a Riemannian submanifold of the Euclidean space $\mathbb{R}^n$. When $p = n\geq 2$, the Stiefel manifold reduces to the orthogonal group $\mathrm{O}(n)$ endowed with the Frobenius metric scaled by $\beta$. In both cases, the length of the shortest geodesic loops is well known, see, e.g.,~\cite[Chap. 6]{absilmataigne2024ultimate}.

Let $2\leq p\leq n-1$. 
In this section, we consider $\beta\in(0,2]$ 
and show that the $\beta$-length of shortest geodesic loops on $\mathrm{St}_\beta(n,p)$ is $\sqrt{2\beta}2\pi$ for $\beta<\frac12$, and $2\pi$ for $\beta\geq\frac12$ (see \cref{thm:shortest_beta_geos<=2}). 
The essential observation is that the geodesics are curves of constant Frenet curvatures, when considered as curves in the ambient Euclidean space $\mathbb{R}^{n\times p}\cong \mathbb{R}^{np}$, see~\cite{zimmermannstoye_eucl_inj:2024}.
As such a Euclidean transformation sends them to the form
\begin{equation}
    \label{eq:curvconstcurv_even}
    \gamma(t) = \left(a_1\cos(b_1t),a_1\sin(b_1t), \ldots, a_m\cos(b_mt),a_m\sin(b_mt)\right)^\top\in \mathbb{R}^{np},
\end{equation}
if $np=2m$ is even and to 
\begin{equation}
    \label{eq:curvconstcurv_odd}
    \gamma(t) = \left(a_1\cos(b_1t),a_1\sin(b_1t), \ldots, a_m\cos(b_mt),a_m\sin(b_mt), a_{m+1} t\right)^\top\in \mathbb{R}^{np},
\end{equation}
if $np=2m+1$ is odd. For loops, necessarily $a_{m+1}=0$, so that we are left with the form~\eqref{eq:curvconstcurv_even}. When viewing the geodesics in this Frenet form, we see that all geodesic loops on the Stiefel manifold $\mathrm{St}_\beta(n,p)$ are closed.
This means that the geodesic not only begins and ends at the same point, but also that the geodesic's tangent at the end point matches the tangent at the starting point.
For background on Frenet curves, consult~\cite[Sections 1.2, 1.3]{Klingenberg:1978} or~\cite[Chapter 2]{kuhnel:2015}.

For a geodesic loop $\gamma$, we let $t_L$ be the period of the loop, i.e., $\gamma(t_L) = \gamma(0)$ and $\gamma(t) \neq \gamma(0)$ for all $t\in(0,t_L)$. Since the length of a loop does not depend on its parameterization, it follows that the unit-speed assumptions in the following results causes no loss of generality and ensures that $t_L$ is equal to the Euclidean length of the loop.

\subsection{The Euclidean length of geodesic loops}

We derive a lower bound on the Euclidean length of geodesic loops on $\mathrm{St}_\beta(n,p)$. This bound is valid for all $\beta > 0$, but it is only useful for deriving a bound on $\ell_\beta$ for $\beta\in(0,2]$ (see \Cref{sec:closedgeod_beta<=2}).
We start with a general length bound.

\begin{lemma}\label[lemma]{lem:Eucl_length}
    Let $\gamma$ be a unit-speed loop with constant Frenet curvatures in the form of~\eqref{eq:curvconstcurv_even}.
    Let $b_{\min} = \min\{|b_i|\mid b_i\neq 0\}$, i.e., the smallest non-zero $b$-coefficient in~\eqref{eq:curvconstcurv_even}.
    The Euclidean length of $\gamma$ satisfies
    \[
      L_\mathrm{E}(\gamma) = t_L \geq \frac{2\pi}{b_{\min}} \geq \frac{2\pi}{\|\ddot \gamma\|_\mathrm{E}}.
    \]
\end{lemma}

\begin{proof}
    Since the loop is of unit speed, we have $1=\|\dot{\gamma}\|_\mathrm{E}^2 = \sum_{i=1} a_i^2b_i^2$. From the periodicity of the sines and cosines in~\eqref{eq:curvconstcurv_even}, it is clear that the curve $\gamma$ closes if and only if there exist $t_L>0$ and $l_i\in\mathbb{Z}\setminus\{0\}$ such that $ 
    t_L = \frac{2\pi l_i}{b_i}$ for all $i \text{ with } b_i\neq 0$. The smallest such $t_L$ corresponds to the Euclidean length of $\gamma$: $L_\mathrm{E}(\gamma)=t_L$.

    Let $b_{\min}:= \min \{|b_i|\mid b_i\neq 0\}>0$ with associated $l_{\min}\in\mathbb{N}$, so that 
\begin{equation*}
     t_L = \frac{2\pi l_{\min}}{b_{\min}}\geq \frac{2\pi}{b_{\min} }.
\end{equation*}
Let $\kappa=\|\ddot \gamma\|_\mathrm{E}$ be the Euclidean curvature of $\gamma$.
Then 
\begin{equation*}
  \kappa^2 = \|\ddot \gamma\|_\mathrm{E}^2 = \sum_i a_i^2b_i^4 \geq b_{\min}^2\sum_i a_i^2 b_i^2 = b_{\min}^2 \Longleftrightarrow 
\frac{1}{b_{\min}} \geq \frac{1}{\|\ddot \gamma\|_\mathrm{E}}.  
\end{equation*}
As a consequence, $ L_\mathrm{E}(\gamma) = t_L \geq \frac{2\pi}{b_{\min} }\geq  \frac{2\pi}{\|\ddot \gamma\|_\mathrm{E}}$.
\end{proof}

Applying this result to the Stiefel geodesics yields the following lemma.

\begin{lemma}\label[lemma]{lem:Eucl_length_beta}
    Recall that $2\leq p\leq n-1$. Let $\gamma_\beta(t)\coloneq \mathrm{Exp}_{\beta,I_{n\times p}}\left(t\begin{bmatrix}
        A\\B
    \end{bmatrix}\right)$ be a geodesic loop in $\mathrm{St}_\beta(n,p)$ of unit Euclidean speed, i.e., $\sqrt{\|A\|_\mathrm{F}^2 + \|B\|_\mathrm{F}^2} = 1$, then the norm squared of the second derivative is bounded by
    \begin{align*}
        \|\ddot{\gamma}_\beta\|_E^2 \leq 1 + (1-4\beta+2\beta^2)\|B\|_\mathrm{F}^2\|A\|_\mathrm{F}^2 -\frac12 \|A\|_\mathrm{F}^4.
    \end{align*}
    As a direct consequence, the Euclidean length of $\gamma_\beta$ satisfies
    \begin{align*}
        L_\mathrm{E}(\gamma_\beta)\geq \frac{2\pi}{\sqrt{1 + (1-4\beta + 2\beta^2)\|B\|_\mathrm{F}^2\|A\|_\mathrm{F}^2 - \frac12\|A\|_\mathrm{F}^4}}.
    \end{align*}
\end{lemma}
\begin{proof}
    The first and second derivatives of the geodesic are given by
    \begin{align*}
        \dot{\gamma}_\beta(t)  &= \exp_{\mathrm{m}}\left(t\begin{bmatrix}
        2\beta A & -B^\top\\
        B & 0
    \end{bmatrix}\right)\begin{bmatrix}
        A\\
        B
    \end{bmatrix} \exp_{\mathrm{m}}(t(1-2\beta)A),\quad \text{and},\\ 
    \ddot\gamma_\beta(t) &=
    \exp_{\mathrm{m}}\left(t\begin{bmatrix}
        2\beta A & -B^\top\\
        B & 0
    \end{bmatrix}\right)\begin{bmatrix}
        -A^\top A-B^\top B\\
        (2-2\beta)BA
    \end{bmatrix} \exp_{\mathrm{m}}(t(1-2\beta)A).
    \end{align*}
    Mind that $A$ is skew-symmetric. For the norm of the second derivative, it holds
    \begin{align*}
        \|\ddot{\gamma}_\beta\|_\mathrm{E}^2 &= \left\|\begin{bmatrix}
            -A^\top A - B^\top B\\
            (2-2\beta)BA
        \end{bmatrix}\right\|_\mathrm{F}^2\\
        & = \mathrm{tr}(A^4) + \mathrm{tr}((B^\top B)^2) + 2 \underbrace{\mathrm{tr}(A^\top A B^\top B)}_{=\mathrm{tr}(A^\top B^\top B A)} + (2-2\beta)^2\mathrm{tr}(A^\top B^\top BA)\\
        &= \|A^\top A\|_\mathrm{F}^2 + \mathrm{tr}((B^\top B)^2) + (6 - 8\beta + 4\beta^2)\|BA\|_\mathrm{F}^2.
    \end{align*}
    Next, we determine an upper bound on the norm of the second derivative that is only dependent on $\beta$, $\|A\|_\mathrm{F}$ and $\|B\|_\mathrm{F}$. 
    By the skew-symmetry of $A$, it holds $\|A^\top A\|_\mathrm{F}\leq \frac12 \|A\|_\mathrm{F}^4$ and $\|BA\|_\mathrm{F}^2\leq\frac12\|B\|_\mathrm{F}^2\|A\|_\mathrm{F}^2$, see~\cite[Lemma 3]{zimmermannstoye_curvature:2024}. 
    The $\beta$-parabola $p(\beta) = (6 - 8\beta + 4\beta^2)$ is always positive, so that we obtain
    \begin{align*}
        \|\ddot \gamma_\beta\|_\mathrm{E}^2 &\leq 
    \frac12\|A\|_\mathrm{F}^4 + \|B\|_\mathrm{F}^4 + 
     (6-8\beta+4\beta^2) \frac12\|B\|_\mathrm{F}^2 \|A\|_\mathrm{F}^2\\
    &= -\frac12 \|A\|_\mathrm{F}^4  + \|A\|_\mathrm{F}^4 + \|B\|_\mathrm{F}^4 + 2\|B\|_\mathrm{F}^2\|A\|_\mathrm{F}^2 
    + (1-4\beta+2\beta^2)\|B\|_\mathrm{F}^2\|A\|_\mathrm{F}^2\\
    &= 1 + (1-4\beta+2\beta^2)\|B\|_\mathrm{F}^2\|A\|_\mathrm{F}^2 -\frac12 \|A\|_\mathrm{F}^4.
    \end{align*}
    Applying \cref{lem:Eucl_length} concludes the proof.
\end{proof}

We record an easy corollary, which is extended in \cref{thm:shortest_beta_geos<=2} to $\beta \in (0,2]$.
\begin{corollary}
    For $\beta\in[1,1+\frac{1}{\sqrt{2}}]$, the $\beta$-length of a shortest geodesic loop on the Stiefel manifold $\mathrm{St}_\beta(n,p)$ is given by $\ell_\beta = 2\pi$.
\end{corollary}
\begin{proof}
    The $\beta$-parabola $p(\beta) = (1-4\beta + 2\beta^2)$ is non-positive in the range $\beta\in[1-\frac{1}{\sqrt{2}},1+\frac{1}{\sqrt{2}}]$. From \cref{lem:Eucl_length_beta}, we obtain that $\|\ddot \gamma_\beta\|_\mathrm{E}\leq 1$ and $ L_\mathrm{E}(\gamma_\beta) \geq 2\pi$ holds in this range. 
    For $\beta \geq 1$, we can now conclude that
    \begin{align*}
        L_\beta(\gamma_\beta) = t_L\sqrt{\beta\|A\|_\mathrm{F}^2 + \|B\|_\mathrm{F}^2}\geq t_L\sqrt{\|A\|_\mathrm{F}^2 + \|B\|_\mathrm{F}^2} = L_\mathrm{E}(\gamma_\beta)\geq 2\pi.
    \end{align*}
    This bound is sharp by the examples of~\cite[Section 6]{absilmataigne2024ultimate}.
\end{proof}

\subsection{The \texorpdfstring{$\beta$}{beta}-length of shortest geodesic loops (\texorpdfstring{$\beta\in (0,2]$}{beta in (0,2]})}
\label{sec:closedgeod_beta<=2}
In this section, we translate the Euclidean length bounds from \cref{lem:Eucl_length_beta} to the $\beta$-metric.  
Consider the geodesic loop $\gamma_\beta\colon [0,t_L]\to\mathrm{St}_\beta(n,p)$ with $\gamma_\beta(t) = \mathrm{Exp}_{\beta,I_{n\times p}}\left(t\begin{bmatrix}
    A\\B
\end{bmatrix}\right)$ of unit Euclidean speed. We have, 
$$\|A\|_\mathrm{F}^2 + \|B\|_\mathrm{F}^2 = 1\Longleftrightarrow \|B\|_\mathrm{F}^2 = 1-\|A\|_\mathrm{F}^2.$$
Recall that $t_L$ is assumed to be the period of the loop and that the Euclidean length of the loop is given by $t_L$. 
When measured according to the $\beta$-metric, we have
\begin{eqnarray*}
    L_\beta(\gamma_\beta)^2 =& \left(\int_0^{t_L}\|\dot{\gamma}_\beta(t)\|_\beta \mathrm{d}t\right)^2 
    &= t_L^2(\beta\|A\|_{\mathrm{F}}^2 + \|B\|_{\mathrm{F}}^2)\\
    =& L_\mathrm{E}(\gamma_\beta)^2(\beta\|A\|_{\mathrm{F}}^2 + 1 - \|A\|_{\mathrm{F}}^2)
    &= L_\mathrm{E}(\gamma_\beta)^2(1 + (\beta-1)\|A\|_{\mathrm{F}}^2).
\end{eqnarray*}
\cref{lem:Eucl_length_beta} yields
\begin{align}\label{eq:para_ratio}
\begin{split}
    L_\beta(\gamma_\beta)^2 &\geq (2\pi)^2\frac{(1+(\beta-1)\|A\|_{\mathrm{F}}^2)}{1+(1 - 4\beta + 2\beta^2)\|B\|_{\mathrm{F}}^2 \|A\|_{\mathrm{F}}^2 - \frac12\|A\|_{\mathrm{F}}^4}\\
    &=(2\pi)^2\frac{(1+(\beta-1)\|A\|_{\mathrm{F}}^2)}{1+(1-4\beta+2\beta^2)\|A\|_{\mathrm{F}}^2 + (-\frac32 + 4\beta-2\beta^2)\|A\|_{\mathrm{F}}^4}.
\end{split}
\end{align}

\begin{theorem}\label{thm:shortest_beta_geos<=2}
    Under the standing assumption that $2\leq p\leq n-1$ and $\beta \in (0,2]$, the $\beta$-length $\ell_\beta$ of a shortest geodesic loop on the Stiefel manifold $\mathrm{St}_\beta(n,p)$ is given by
    \begin{align*}
        \ell_\beta = \min\{\sqrt{2\beta},1\}2\pi.
    \end{align*}
\end{theorem}
\begin{proof}
    There are geodesic loops on $\mathrm{St}_\beta(n,p)$ that have the length $\sqrt{2\beta}2\pi$ or $2\pi$, see~\cite[Section 6]{absilmataigne2024ultimate}. Hence, we are left with proving that this upper bound $\ell_\beta\leq\min\{\sqrt{2\beta},1\}2\pi$ on the length of a shortest geodesic loop is also a lower bound. 
    For $\beta\in[\frac12,2]$, the upper bound is given by $2\pi$ and, for $\beta\in(0,\frac12)$, the upper bound is given by $\sqrt{2\beta}2\pi$.
    It is sufficient to show that the parametric fraction in~\eqref{eq:para_ratio} is greater than or equal to $1$ for all $\|A\|_{\mathrm{F}}\in[0,1]$ and $\beta\in[\frac12,2]$ and that it is greater than or equal to $2\beta$ for all $\|A\|_{\mathrm{F}}\in[0,1]$ and $\beta\in(0,\frac12)$.

    We start by showing that for all $\|A\|_{\mathrm{F}}\in[0,1]$ and $\beta\in[\frac12,2]$, we have
    \begin{align*}
        &&\frac{(1+(\beta-1)\|A\|_{\mathrm{F}}^2)}{1+(1-4\beta+2\beta^2)\|A\|_{\mathrm{F}}^2 + (-\frac32 + 4\beta-2\beta^2)\|A\|_{\mathrm{F}}^4}&\geq 1.
    \end{align*}
    This sought conclusion is successively equivalent to
    \begin{align}
        1+(1-4\beta+2\beta^2)\|A\|_{\mathrm{F}}^2 + (-\frac32 + 4\beta-2\beta^2)\|A\|_{\mathrm{F}}^4 &\leq 1 + (\beta-1)\|A\|_{\mathrm{F}}^2,\notag\\
        \left(\beta-\frac12\right)\|A\|_{\mathrm{F}}^2\left(2(\beta-2) + (3-2\beta)\|A\|_{\mathrm{F}}^2\right) &\leq 0.\label{eq:shortest_loop_case_1/2to2}
    \end{align}
    Since $\beta\in[\frac12,2]$, it follows that the first factor is non-negative and that the third factor is non-positive, being an affine function of $\|A\|_{\mathrm{F}}^2$ that is non-positive at the endpoints $\|A\|_\mathrm{F}^2 = 0$ and $\|A\|_\mathrm{F}^2 = 1$ for all $\beta$ in the considered range. So, the inequality in~\eqref{eq:shortest_loop_case_1/2to2} is fulfilled for all $\|A\|_{\mathrm{F}}\in[0,1]$ and $\beta \in[\frac12,2]$.
    This shows that for every geodesic loop $\gamma_\beta$ in $\mathrm{St}_\beta(n,p)$ with $\beta\in [\frac12, 2]$, its length is bounded below by $2\pi$, i.e., 
    $$L_\beta(\gamma_\beta) \geq 2\pi.$$ 
    Therefore, we have $\ell_\beta = 2\pi$, for $\beta\in[\frac12,2]$.

    It remains to show that for all $\|A\|_{\mathrm{F}}\in [0,1]$ and $\beta\in(0,\frac12)$, we have
    \begin{align}\label{eq:fraction_loop_case_0to1/2}
        \frac{\left(1 + (\beta - 1)\|A\|_{\mathrm{F}}^2\right)}{1 + (1 - 4\beta + 2\beta^2)\|A\|_{\mathrm{F}}^2 + (-\frac32 + 4\beta - 2\beta^2)\|A\|_{\mathrm{F}}^4}\geq 2\beta.
    \end{align}
    This inequality is successively equivalent to
    \begin{align}
         1 + (\beta - 1)\|A\|_{\mathrm{F}}^2    &\geq 2\beta + 2\beta(1 - 4\beta + 2\beta^2)\|A\|_{\mathrm{F}}^2 + 2\beta(-\frac32 + 4\beta - 2\beta^2)\|A\|_{\mathrm{F}}^4, \notag\\
         0        &\geq 2\beta - 1 + (1 + \beta - 8\beta^2 + 4\beta^3)\|A\|_{\mathrm{F}}^2 + (-3\beta + 8\beta^2 - 4\beta^3)\|A\|_{\mathrm{F}}^4, \notag\\
         0                &\leq (\|A\|_{\mathrm{F}}^2 - 1)(2\beta - 1)(\|A\|_{\mathrm{F}}^2\beta(2\beta - 3) + 1).\notag
    \end{align}
    The first factor is non-positive. For a fixed $\beta$, the remaining two factors together form an affine function of $\|A\|_\mathrm{F}^2$. As before, we have to check whether the values at the two ends $\|A\|_\mathrm{F}^2 = 0$ and $\|A\|_\mathrm{F}^2 = 1$ are non-positive. It is easy to verify that this holds true for $\beta\in(0,\frac12)$. Hence, the inequality in~\eqref{eq:fraction_loop_case_0to1/2} is fulfilled for all $\|A\|_{\mathrm{F}}\in[0,1]$ and $\beta\in(0,\frac12)$. This shows that for a geodesic loop $\gamma_\beta$ in $\mathrm{St}_\beta(n,p)$ with $\beta\in(0,\frac12)$, it holds
    \begin{align*}
        L_\beta(\gamma_\beta) \geq \sqrt{2\beta}2\pi. 
    \end{align*}
    Consequently, we have $\ell_\beta = \sqrt{2\beta}2\pi$, for $\beta\in(0,\frac12)$.
\end{proof}

Note that this result includes the earlier ones for the canonical and Euclidean metric of~\cite{Rentmeesters2013,absilmataigne2024ultimate} and~\cite{zimmermannstoye_eucl_inj:2024}, respectively.

\section{The exponential inverse and invariance groups}\label{sec:exp_inverse_and_ig}
The proof technique of \cref{thm:shortest_beta_geos<=2} cannot be extended to the case $\beta> 2$. Indeed, the left-hand side of~\eqref{eq:shortest_loop_case_1/2to2} can become positive for some values of $\|A\|_{\mathrm{F}}\in[0,1]$.
In order to prove the length of the shortest geodesic loops on the Stiefel manifold for $\beta>2$, we need to manipulate matrix exponential inverses. To this end, we present five short lemmata. 

Let us define the inverse of the matrix exponential $\exp_\mathrm{m}:\mathrm{Skew}(p)\rightarrow \mathrm{SO}(p)$ as the set-valued function
$$\exp_\mathrm{m}^{-1}(Q)\coloneq \{ \Omega\in\mathrm{Skew}(p)\ | \ \exp_\mathrm{m}(\Omega) = Q\}. $$
A very important set of orthogonal matrices are the real Schur forms of $\mathrm{SO}(p)$. Let $k\coloneq \lfloor \frac{p}{2} \rfloor$, $J_2=\left[\begin{smallmatrix}
    0&-1\\
    1&0
\end{smallmatrix}\right]$ and $J_{2p} = J_2\otimes I_p $. For all diagonal matrices $\Phi \in \mathbb{R}^{k\times k}$, if $p$ is even, then $$G_{p}(\Phi)\coloneq \cos(\Phi)\otimes I_2 + \sin(\Phi)\otimes J_2 =\begin{bmatrix} \cos\Phi_{11} & -\sin\Phi_{11} & & & \\ \sin\Phi_{11} & \cos\Phi_{11} & & & \\ & & \ddots & & \\ & & & \cos\Phi_{kk} & -\sin\Phi_{kk} \\ & & & \sin\Phi_{kk} & \cos\Phi_{kk} \end{bmatrix}\in \mathrm{SO}(2k)=\mathrm{SO}(p).$$ 
Otherwise, if $p$ is odd we have 
\begin{equation*}
    G_{p}(\Phi)\coloneq \begin{bmatrix}
    \cos(\Phi)\otimes I_2 + \sin(\Phi)\otimes J_2&0\\
    0&1
\end{bmatrix}\in\mathrm{SO}(p).
\end{equation*}
 Similarly, we define the real Schur forms of $\mathrm{Skew}(p)$. For all diagonal matrices $\Phi \in \mathbb{R}^{k\times k}$ and $K\in\mathbb{Z}^{k\times k}$, if $p$ is even, then $\Omega_{p}(\Phi, K)\coloneq (\Phi + 2\pi K)\otimes J_2$. Otherwise, if $p$ is odd we have 
\begin{equation*}
    \Omega_{p}(\Phi, K)\coloneq \begin{bmatrix}
    (\Phi+2\pi K)\otimes J_2&0\\
    0&0
\end{bmatrix}\in\mathrm{Skew}(p).
\end{equation*}

\cref{lem:diagonal_invariance,lem:invarianceAB} describe the structure of the invariance groups of matrices with disjoint spectra. The \emph{invariance group} of $A\in\mathbb{R}^{k\times k}$ is $$\mathrm{ig}(A)\coloneq \{M\in\mathrm{O}(k)\ | \ A = MAM^\top\}.$$ The \emph{complex invariance group} of $A\in\mathbb{C}^{k\times k}$ is $\mathrm{cig}(A)\coloneq \{M\in\mathrm{U}(k)\ | \ A = MAM^*\}$. The invariance groups are fundamental to analyze the matrix exponential inverse. In particular, a matrix $M$ is termed \emph{$J_{2n} $-orthosymplectic} if $ M\in \mathrm{ig}(J_{2n})$.
\begin{lemma}\label[lemma]{lem:diagonal_invariance}
	Given two diagonal matrices $\Lambda_A,\Lambda_B\in\mathbb{C}^{n\times n}$ with disjoint spectra, it holds that
	\begin{equation*}
	\mathrm{cig}\left(\begin{bmatrix}
	\Lambda_A&0\\
	0&\Lambda_B
	\end{bmatrix}\right) = \begin{bmatrix}
	\mathrm{cig}(\Lambda_A)&0\\
	0&\mathrm{cig}(\Lambda_B)
	\end{bmatrix}.
	\end{equation*}
\end{lemma}
\begin{proof}
	Letting $M\coloneq \left[\begin{smallmatrix}
	M_{11}&M_{12}\\
	M_{21}&M_{22}
	\end{smallmatrix}\right]\in \mathrm{cig}\left(\left[\begin{smallmatrix}
	\Lambda_A&0\\
	0&\Lambda_B
	\end{smallmatrix}\right]\right)$, it must hold by definition of invariance groups that
	\begin{equation}\label{eq:M12}
	M_{12}\Lambda_B = \Lambda_A M_{12} \quad \text{and}\quad M_{21}\Lambda_A = \Lambda_B M_{21}.
	\end{equation}
	Since $\Lambda_A$ and $\Lambda_B$ have disjoint spectra,~\eqref{eq:M12} is true if and only if $M_{12}=M_{21}=0$. Therefore, $M$ is block diagonal and this implies $M_{11}\in\mathrm{cig}(A)$ and $M_{22}\in\mathrm{cig}(B)$, which concludes the proof. 
\end{proof}
The result about diagonal matrices in \cref{lem:diagonal_invariance} can be generalized to normal matrices.
\begin{lemma}\label[lemma]{lem:invarianceAB}
	Given two normal matrices $A,B\in\mathbb{R}^{n\times n}$ with disjoint spectra, it holds that
	\begin{equation*}
	\mathrm{ig}\left(\begin{bmatrix}
	A&0\\
	0&B
	\end{bmatrix}\right) = \begin{bmatrix}
	\mathrm{ig}(A)&0\\
	0&\mathrm{ig}(B)
	\end{bmatrix}.
	\end{equation*}
\end{lemma}
\begin{proof}
	Let $A= V_A \Lambda_A V_A^*$ and $B= V_B \Lambda_B V_B^*$ be eigenvalue decompositions. $V_A$ and $V_B$ are unitary since $A$ and $B$ are normal.
	For all $M\in \mathrm{ig}\left(\left[\begin{smallmatrix}
	A&0\\
	0&B
	\end{smallmatrix}\right]\right)$, we have 
	\begin{equation}\label{eq:invariance_AB}
	M\begin{bmatrix}
	A&0\\
	0&B
	\end{bmatrix} M^\top=\begin{bmatrix}
	A&0\\
	0&B
	\end{bmatrix} \iff \widetilde{M}\begin{bmatrix}
	\Lambda_A&0\\
	0&\Lambda_B
	\end{bmatrix}\widetilde{M}^*=\begin{bmatrix}
	\Lambda_A&0\\
	0&\Lambda_B
	\end{bmatrix},
\end{equation}	
where $\widetilde{M}= \left[\begin{smallmatrix}
	V_A^*&0\\
	0&V_B^*
	\end{smallmatrix}\right]M \left[\begin{smallmatrix}
	V_A&0\\
	0&V_B
	\end{smallmatrix}\right]$. Therefore, by \cref{lem:diagonal_invariance}, $\widetilde{M}$ has a block structure $\widetilde{M}= \left[\begin{smallmatrix}
	\widetilde{M}_{11}&0\\
	0&\widetilde{M}_{22}
	\end{smallmatrix}\right]$ and so does $M= \left[\begin{smallmatrix}
	V_A^*\widetilde{M}_{11} V_A&0\\
	0&V_B^*\widetilde{M}_{22} V_B
	\end{smallmatrix}\right]$. By~\eqref{eq:invariance_AB}, we have $M_{11}\in\mathrm{ig}(A)$ and $M_{22}\in\mathrm{ig}(B)$. This concludes the proof.
\end{proof}
Next, we show how the invariance groups affect the computation of the matrix exponential inverse.
\begin{lemma}\label[lemma]{lem:logGPhiGeneral}
For every positive integer $p$, $k\coloneq \lfloor \frac{p}{2}\rfloor$ and all diagonal matrices  $\Phi \in \mathbb{R}^{k\times k}$, we have
\begin{equation*}
	\exp_\mathrm{m}^{-1}(G_p(\Phi)) = \left\{ M \Omega_p(\Phi, K) M^\top 
\   | \ K\in\mathbb{Z}^{k\times k} \text{ is diagonal},\ M\in\mathrm{ig}(G_p(\Phi))
\right\}.
\end{equation*} 
\end{lemma}
\begin{proof}
	The spectrum of $G_p(\Phi)$ imposes that every matrix in $\exp_\mathrm{m}^{-1}(G_p(\Phi))$ is an orthogonal similarity transformation of $\Omega_p(\Phi,K)$ for some $K\in\Z^{k\times k}$, therefore, $ 
	\exp_\mathrm{m}^{-1}(G_p(\Phi)) \subseteq \{ M\Omega_p(\Phi,K)M^\top
\   | \ K\in\mathbb{Z}^{k\times k}\text{ is diagonal}
, M \in\mathrm{O}(2k)\}$. Finally, $ M\Omega_p(\Phi,K)M^\top\in \exp_\mathrm{m}^{-1}(G_p(\Phi))$ if and only if
	\begin{equation*}
	G_p(\Phi) = \exp_\mathrm{m} (M\Omega_p(\Phi,K)M^\top)= M\exp_\mathrm{m} (\Omega_p(\Phi,K))M^\top = M G_p(\Phi)M^\top,
\end{equation*}
i.e., $M\in \mathrm{ig}(G_p(\Phi))$.
\end{proof}

The next lemma is a corollary of \cref{lem:logGPhiGeneral} for the particular case where the diagonal matrix $\Phi$ is the identity $I_k$ multiplied by $\varphi \in (-\pi,\pi)\setminus\{0\}$. In this case, $\mathrm{ig}(G_{2k}(\varphi I_k))$ is the very special set of $J_{2k}$-orthosymplectic matrices.

\begin{lemma}\label[lemma]{lem:logGphiI} 
For every $\varphi \in (-\pi,\pi)\setminus\{0\}$, $k\in\mathbb{N}_0$, we have
\begin{align*}
	\exp_\mathrm{m}^{-1}(G_{2k}(\varphi I_k)) =& \Big\{M \Omega_{2k}(\varphi I_k, K)M^\top
\   | \ K\in\mathbb{Z}^{k\times k},\\ &K \text{ is diagonal and } M \text{ is $J_{2k}$-orthosymplectic}
\Big\}.
\end{align*}
\end{lemma}
\begin{proof}
	Notice that $\mathrm{ig}(G_{2k}(\varphi I_k))=\{M \text{ is  $J_{2k}$-orthosymplectic}\}$. 
    Since $\mathrm{ig}(G_{2k}(\varphi I_k))\not\subset \mathrm{ig} \left( \Omega_{2k}(\varphi I_k, K)\right)$, by \cref{lem:logGPhiGeneral} in general, this concludes the proof.
\end{proof}

Finally, we relate the matrix exponential inverse of a block diagonal matrix to the block diagonal matrix of the matrix exponential inverses.
\begin{lemma}\label[lemma]{lem:logmatrix}
Let $k\coloneq\lfloor\frac{p}{2}\rfloor $. For all positive integers $k_1,p$, for all $\varphi\in (0,\pi]$ and all nonnegative diagonal matrices $\Phi\in\mathbb{R}^{k\times k}$, $\|\Phi\|_2\leq\pi$, such that $\varphi$ is not a diagonal entry of $\Phi$, we have 
\begin{equation*}
	\exp_\mathrm{m}^{-1}\begin{bmatrix}
	G_{2k_1}(\varphi I_{k_1})&0\\
	0&G_p(\Phi)
	\end{bmatrix}=\begin{bmatrix}
	\exp_\mathrm{m}^{-1} (G_{2k_1}(\varphi I_{k_1}))&0\\
	0&\exp_\mathrm{m}^{-1}(G_p(\Phi))
	\end{bmatrix}.
\end{equation*}
\end{lemma}

\begin{proof}
For every diagonal matrix $K_1\in \mathbb{Z}^{k_1\times k_1}$ and $K\in\mathbb{Z}^{k\times k}$, it holds that
\begin{equation*}
	\begin{bmatrix}

\Omega_{2k_1}(\varphi I_{k_1}, K_1)&0\\
0&\Omega_p (\Phi, K)
\end{bmatrix}\in \exp_\mathrm{m}^{-1}\begin{bmatrix}
	G_{2k_1}(\varphi I_{k_1})&0\\
	0&G_p(\Phi)
	\end{bmatrix}\cap \begin{bmatrix}
	\exp_\mathrm{m}^{-1} (G_{2k_1}(\varphi I_{k_1}))&0\\
	0&\exp_\mathrm{m}^{-1}(G_p(\Phi))
	\end{bmatrix}.
\end{equation*}
By \cref{lem:logGPhiGeneral}, it is thus enough to prove that $\mathrm{ig}\left(\left[\begin{smallmatrix}
	G_{2k_1}(\varphi I_{k_1})&0\\
	0&G_p(\Phi)
	\end{smallmatrix}\right]\right)= \left[\begin{smallmatrix}
	\mathrm{ig}(G_{2k_1}(\varphi I_{k_1}))&0\\
	0&\mathrm{ig}(G_p(\Phi))
	\end{smallmatrix}\right]$. The later holds by \cref{lem:invarianceAB}. This concludes the proof.
\end{proof}
\section{The \texorpdfstring{$\beta$}{beta}-length of geodesic loops for \texorpdfstring{$\beta\in(2,\infty)$}{beta in (2,oo)}}
\label{sec:closedgeod_beta>2}
As in \Cref{sec:betalength<=2}, we restrict to $2\leq p\leq n-1$. In this section, we show that the length of a shortest geodesic loop on $\mathrm{St}_\beta(n,p)$, for $\beta\in(2,\infty)$, is $2\pi$, see \cref{thm:closedgeod_beta>2}. We proceed in two phases. In the first phase, we study the structure of shortest geodesic loops to simplify the problem as much as possible. This is \Cref{sec:structured_loop}. In the second phase, we analyze the simplified problem to prove that shortest geodesic loops have length $2\pi$. This is \Cref{sec:shortest_loops_g2}.

\subsection{The structure of shortest geodesic loops}\label{sec:structured_loop}
We first prove that every lower bound on the length of the shortest geodesic loops on $\mathrm{St}_\beta(2p,p)$ is a lower bound on the length of shortest geodesic loops on $\mathrm{St}_\beta(n,p)$. We can thus concentrate on geodesic loops on the Stiefel manifold with $n=2p$.
\begin{lemma}\label[lemma]{lem:nxp_to_2pxp}
    Recall that $2\leq p \leq n - 1$ and let $\ell_\beta$ be the length of the shortest geodesic loop starting at $I_{n\times p}$ in $\mathrm{St}_\beta(n,p)$ and $\widehat{\ell}_\beta$ the length of the shortest geodesic loop starting at $I_{2p\times p}$ in $\mathrm{St}_\beta(2p,p)$. Then, $\ell_\beta \geq \widehat{\ell}_\beta$.
\end{lemma}
\begin{proof}
    Assume first $n\geq 2p$. For every tangent vector $\Delta =\left[\begin{smallmatrix}
        I_p\\
        0_{n-p\times p}
    \end{smallmatrix}\right]A + \left[\begin{smallmatrix}
        0_{p\times n-p}\\
        I_{n-p}
    \end{smallmatrix}\right] B\in T_{I_{n\times p}}\mathrm{St}(n,p)$, with $B\in\mathbb{R}^{(n-p)\times p}$, a thin QR decomposition yields $ B = QR$ with $Q\in\mathrm{St}(n-p,p)$ and $R\in\mathbb{R}^{p\times p}$. Therefore \begin{equation*}
        \gamma_\beta(t)\coloneq \mathrm{Exp}_{\beta,I_{n\times p}}\left(t\begin{bmatrix}
                A\\
                B
            \end{bmatrix}\right) = \begin{bmatrix}
                I_p&0\\
                0&Q
            \end{bmatrix} \mathrm{Exp}_{\beta,I_{2p\times p}}\left(t\begin{bmatrix}
                A\\
                R
            \end{bmatrix}\right)\eqcolon \begin{bmatrix}
                I_p&0\\
                0&Q
            \end{bmatrix} \widehat{\gamma}_\beta(t).
    \end{equation*}
    Moreover, since $\|B\|_\mathrm{F}=\|R\|_\mathrm{F}$, we have $L_\beta(\gamma_\beta)=L_\beta(\widehat{\gamma}_\beta)$ for all $\Delta\in\mathrm{St}(n,p)$ and it holds that $\ell_\beta\geq \widehat{\ell}_\beta$ if $n\geq 2p$.
    
    Assume now that $n<2p$. Notice that for every geodesic $\gamma_\beta$ of $\mathrm{St}_\beta(n,p)$, the padded curve
    \begin{equation*}
        \widehat{\gamma}_\beta(t)\coloneq\begin{bmatrix}
            \gamma_\beta(t)\\
            0_{(2p-n)\times p}
        \end{bmatrix},
    \end{equation*}
    is a geodesic of $\mathrm{St}_\beta(2p,p)$. Finally, since $L_\beta(\gamma_\beta)=L_\beta(\widehat{\gamma}_\beta)$, it is clear that $\ell_\beta \geq \widehat{\ell}_\beta$.
\end{proof}
By \cref{lem:nxp_to_2pxp}, a lower bound on the length $\widehat{\ell}_\beta$ of shortest geodesic loops in $\mathrm{St}_\beta(2p,p)$ gives a lower bound on the length $\ell_\beta$ of shortest geodesic loops in $\mathrm{St}_\beta(n,p)$, for $2\leq p\leq n-1$. Therefore, if a geodesic loop in $\mathrm{St}_\beta(n,p)$ has a length equal to $\widehat{\ell}_\beta$, then its length is also equal to $\ell_\beta$. It is thus sufficient to consider geodesic loops on $\mathrm{St}_\beta(2p,p)$ emanating from $I_{2p\times p}$,
\begin{equation}\label{eq:geodesic}
    \gamma_\beta(t)\coloneq \mathrm{Exp}_{\beta, I_{2p\times p}}\left(t\begin{bmatrix}
    A\\
    B
\end{bmatrix}\right)= \exp_\mathrm{m}\left(t\begin{bmatrix}
    2\beta A & -B^\top\\
    B & 0
\end{bmatrix}\right)I_{2p\times p}\exp_\mathrm{m}(t(1-2\beta)A),
\end{equation}
with $A\in\mathrm{Skew}(p)$ and $B\in\mathbb{R}^{p\times p}$. The definition of a geodesic loop is that it starts and ends in the same point. When parameterized on the unit interval, $\gamma_\beta(1) = \gamma_\beta(0)=I_{2p\times p}$. This yields the following equation
\begin{equation}\label{eq:problem_large}
\exp_\mathrm{m}\begin{bmatrix}
2\beta A & - B^\top\\
B        & 0
\end{bmatrix} I_{2p\times p}
\exp_\mathrm{m}((1-2\beta)A) = I_{2p\times p}.
\end{equation}
A shortest geodesic loop features $L(\gamma_\beta) = \sqrt{\beta\Vert A\Vert_\mathrm{F}^2 + \Vert B\Vert_\mathrm{F}^2}$ as small as possible. 
Define
\begin{align*}
    E &\coloneq \exp_\mathrm{m}\begin{bmatrix}
        2\beta A & -B^\top\\
        B & 0
    \end{bmatrix} = \begin{bmatrix}
        E_{11} & E_{12}\\
        E_{21}& E_{22}
    \end{bmatrix} \in \mathrm{SO}(2p),\quad \text{and}\\
    \widetilde{E} &\coloneq \exp_\mathrm{m}((1-2\beta)A)\in \mathrm{SO}(p).
\end{align*}
Now, equation~\eqref{eq:problem_large} holds if and only if
\begin{align*}
    \begin{bmatrix}
        E_{11} & E_{12}\\
        E_{21}& E_{22}
    \end{bmatrix}\begin{bmatrix}
        \widetilde{E}\\
        0
    \end{bmatrix} = \begin{bmatrix}
        I_p\\0
    \end{bmatrix},
\end{align*}
which implies $E_{11}\widetilde{E}= I_p$ and $E_{21} \widetilde{E} = 0$. 
Since $\widetilde{E}$ is orthogonal, we obtain $E_{11} = \widetilde{E}^\top\in \mathrm{SO}(p)$ and $E_{12}=E_{21}^\top = 0$. Finally, we have
\begin{align*}
    E = \begin{bmatrix}
        E_{11} & 0\\
        0 & E_{22}
    \end{bmatrix},\quad \text{and} \quad \widetilde{E} = E_{11}^\top,
\end{align*}
with $E_{11}\in \mathrm{SO}(p)$, $E_{22}\in \mathrm{SO}(p)$. 
Both $E_{11}$ and $E_{22}$ admit a real Schur decomposition such that
\begin{equation}\label{eq:schur_forms}
    E_{11} = V_1 G_p(\Phi_1) V_1^\top,\quad \text{and} \quad E_{22} =  V_2 G_p(\Phi_2) V_2^\top,
\end{equation}
where $V_1,V_2\in\mathrm{SO}(p)$ and $\Phi_1,\Phi_2 \in\mathbb{R}^{k\times k}$ are nonnegative diagonal matrices, $k\coloneq \lfloor\frac{p}{2}\rfloor$ and $\|\Phi_1\|_2,\|\Phi_2\|_2\leq \pi$. We now state a lemma on the isometry of curves.
\begin{property}\label[property]{lem:curve_isometry}
For all $\Delta =\left[\begin{smallmatrix}
            A\\
            B
        \end{smallmatrix}\right]\in T_{I_{2p\times p}}\mathrm{St}_\beta (2p,p)$, all $t\in\mathbb{R}$ and all $V_1,V_2\in\mathrm{SO}(p)$, we have
    \begin{equation}\label{eq:curve_isometry}
     \gamma_\beta(t)\coloneq \mathrm{Exp}_{\beta, I_{2p\times p}}\left(t\begin{bmatrix}
            A\\
            B
        \end{bmatrix}\right)= \begin{bmatrix}
            V_1&0\\
            0&V_2
        \end{bmatrix}\mathrm{Exp}_{\beta, I_{2p\times p}}\left(t\begin{bmatrix}
            V_1^\top A V_1\\
            V_2^\top B V_1
        \end{bmatrix}\right) V_1^\top\eqcolon \begin{bmatrix}
            V_1&0\\
            0&V_2
        \end{bmatrix}\widetilde{\gamma}_\beta(t) V_1^\top.
    \end{equation}
    Moreover, $L_\beta(\gamma_\beta) = L_\beta(\widetilde{\gamma}_\beta)$ and if $\gamma_\beta$ is a geodesic loop, so is $\widetilde{\gamma}_\beta$.
\end{property}
\begin{proof}
First, it holds that
\begin{align*}
    \mathrm{Exp}_{\beta, I_{2p\times p}}\left(t\begin{bmatrix}
            V_1^\top A V_1\\
            V_2^\top B V_1
        \end{bmatrix}\right) &= \exp_\mathrm{m}\begin{bmatrix}
2\beta  V_1^\top A V_1 & - ( V_2^\top B V_1)^\top\\
 V_2^\top B V_1        & 0
\end{bmatrix} I_{2p\times p}
\exp_\mathrm{m}((1-2\beta) V_1^\top A V_1) \\
&=  \begin{bmatrix}
            V_1^\top&0\\
            0&V_2^\top 
        \end{bmatrix}\exp_\mathrm{m}\begin{bmatrix}
2\beta A & - B^\top\\
B        & 0
\end{bmatrix} I_{2p\times p}
\exp_\mathrm{m}((1-2\beta)A) V_1 \\
&=  \begin{bmatrix}
            V_1^\top&0\\
            0&V_2^\top 
        \end{bmatrix} \mathrm{Exp}_{\beta, I_{2p\times p}}\left(t\begin{bmatrix}
            A\\
            B
        \end{bmatrix}\right) V_1.
\end{align*}
Moreover,
    \begin{equation*}
        L_\beta (\widetilde{\gamma}_\beta)^2 = \beta \| V_1^\top A V_1\|_\mathrm{F}^2 + \|V_2^\top B V_1\|_\mathrm{F}^2 = \beta \|A\|_\mathrm{F}^2 + \|B\|_\mathrm{F}^2 = L_\beta (\gamma_\beta)^2.
    \end{equation*}
Finally, if $\gamma_\beta$ is a geodesic loop such that $\gamma_\beta(0) = \gamma_\beta(1)=I_{2p\times p}$, by~\eqref{eq:curve_isometry}, we also have $\widetilde{\gamma}_\beta(0) = \widetilde{\gamma}_\beta(1)=I_{2p\times p}$.
\end{proof}

In view of \cref{lem:curve_isometry}, for every geodesic loop $\gamma_\beta$, considering $V_1,V_2$ from~\eqref{eq:schur_forms}, there is another geodesic loop $\widetilde{\gamma}_\beta$, with initial velocity $\widetilde{\Delta} = \left[\begin{smallmatrix}
    \widetilde{A}\\
    \widetilde{B}
\end{smallmatrix}\right]$ that has the same length as $\gamma_\beta $ and such that 
\begin{equation}\label{eq:schur_geodesic}
    \exp_\mathrm{m}\begin{bmatrix}
        2\beta \widetilde{A}&-\widetilde{B}^\top\\
        \widetilde{B}&0
    \end{bmatrix} = \begin{bmatrix}
        G_p(\Phi_1)&0\\
        0&G_p(\Phi_2)
    \end{bmatrix},\quad \text{and}\quad \exp((1-2\beta)\widetilde{A}) = G_p(-\Phi_1).
\end{equation}
For the rest of this section, it is thus sufficient to consider geodesic loops that satisfy the structure of~\eqref{eq:schur_geodesic}.
Moreover, by means of \cref{lem:reduced_geodesic}, we prove that to study shortest geodesic loops, we can consider even more structure. More precisely, it is enough to consider that $\Phi_1$ and $\Phi_2$ in~\eqref{eq:schur_geodesic} are multiples of the identity matrix.
\begin{lemma}\label[lemma]{lem:reduced_geodesic}
Let $p\leq n-1$ and $k\coloneq \lfloor \frac{p}{2}\rfloor $, $A\in\mathrm{Skew}(p)$ and $B\in\mathbb{R}^{p\times p}$, such that
\begin{equation}\label{eq:hypot}
	\exp_\mathrm{m}\begin{bmatrix}
	2\beta A&-B^\top\\
	B&0
	\end{bmatrix}=\begin{bmatrix}
	G_p(\Phi_1)&0\\
	0&G_p(\Phi_2)
	\end{bmatrix}, \quad \text{and} \quad \exp_\mathrm{m}((1-2\beta)A)=G_p(-\Phi_1),
\end{equation}
where $\Phi_1,\Phi_2\in\mathbb{R}^{k \times k}$ are (w.l.o.g., nonnegative) diagonal matrices such that $\|\Phi_1\|_2,\|\Phi_2\|_2\leq\pi$. 
Then, for every $\varphi\neq 0$ that has a multiplicity $s_1\geq 1$ in the vector $\mathrm{diag}(\Phi_1)$ and a multiplicity $s_2\geq 0$ in $\mathrm{diag}(\Phi_2)$, there is a $2s_1\times 2s_1$ submatrix $\widehat{A}$ of $A$ and a $2s_2\times 2s_1$ submatrix $\widehat{B}$ of $B$ such that
\begin{equation}\label{eq:reduced_loop}
	\exp_\mathrm{m}\begin{bmatrix}
	2\beta \widehat{A}&-\widehat{B}^\top\\
	\widehat{B}&0
	\end{bmatrix}=\begin{bmatrix}
	G_{2s_1}(\varphi I_{s_1})&0\\
	0&G_{2s_2}(\varphi I_{s_2})
	\end{bmatrix}, \quad \text{and} \quad \exp_\mathrm{m}((1-2\beta)\widehat{A})=G_{2s_1}(-\varphi I_{s_1}).
\end{equation}
\end{lemma}
\begin{proof}
Assume $\varphi\neq 0$ has a multiplicity $s_1$ in the vector $\Phi_1$ and a multiplicity $s_2$ in $\Phi_2$, and let $m\coloneq s_1+s_2$. Then, there is a permutation $P_{\varphi}$ such that for some $\widetilde{\Phi}\in\mathbb{R}^{(p-m)\times(p-m)}$, we have
	\begin{equation}\label{eq:permutation}
	P_\varphi^\top \begin{bmatrix}
	G_p(\Phi_1)&0\\
	0&G_p(\Phi_2)
	\end{bmatrix}P_\varphi = \begin{bmatrix}
	G_{2s_1}(\varphi I_{s_1})&0&0\\
	0&G_{2s_2}(\varphi I_{s_2})&0\\
	0&0&G_{2(p-m)}(\widetilde{\Phi})
	\end{bmatrix}.
	\end{equation}
    Notice that if $p$ is odd, $\widetilde{\Phi} $ has the diagonal entry zero at least twice. Moreover, it is understood that if $s_2=0$ or $p-m = 0$, then the associated matrix blocks vanish.
    
    By hypothesis~\eqref{eq:hypot}, it follows that
	\begin{equation}\label{eq:skewpermuted}
		\begin{bmatrix}
	2\beta A&-B^\top\\
	B&0
	\end{bmatrix}\in \exp_\mathrm{m}^{-1}\begin{bmatrix}
	G_p(\Phi_1)&0\\
	0&G_p(\Phi_2)
	\end{bmatrix}.
	\end{equation}
By combining~\eqref{eq:permutation} and~\eqref{eq:skewpermuted}, we obtain
	\begin{equation}
		X_\varphi \coloneq P_{\varphi}^\top \begin{bmatrix}
	2\beta A&-B^\top\\
	B&0
	\end{bmatrix} P_{\varphi}  \in \exp_\mathrm{m}^{-1}\begin{bmatrix}
	G_{2s_1}(\varphi I_{s_1})&0&0\\
	0&G_{2s_2}(\varphi I_{s_2})&0\\
	0&0&G_{2(p-m)}(\widetilde{\Phi})
	\end{bmatrix}.
	\end{equation}
	Notice that $\varphi$ is not a diagonal entry of $\widetilde{\Phi}$. Hence, we are able to apply \cref{lem:logmatrix} to obtain that $X_\varphi$ is block-diagonal and that the $2m\times 2m$ upper left block of $X_\varphi$ belongs to the set $\exp_\mathrm{m}^{-1}\left[\begin{smallmatrix}
	G_{2s_1}(\varphi I_{s_1})&0\\
	0&G_{2s_2}(\varphi I_{s_2})
	\end{smallmatrix}\right]$. 
    Moreover, if we require $P_\varphi$ to satisfy that $G_{2s_1}(\varphi I_{s_1})$ is only obtained from matrix elements of $G_p(\Phi_1)$, this upper left block of $X_\varphi$ conserves a structure of the form $\left[\begin{smallmatrix}
	2\beta \widehat{A}& -\widehat{B}^\top\\	
	\widehat{B}&0
	\end{smallmatrix}\right]$ where $\widehat{A}\in \mathrm{Skew}(2s_1)$ and $\widehat{B}\in\mathbb{R}^{2s_2\times 2s_1}$ are, respectively, sub-blocks of $A$ and $B$. We have thus shown that \begin{equation*}   
	\exp_\mathrm{m}\begin{bmatrix}
	2\beta \widehat{A}&-\widehat{B}^\top\\
	\widehat{B}&0
	\end{bmatrix}=\begin{bmatrix}
	G_{2s_1}(\varphi I_{s_1})&0\\
	0&G_{2s_2}(\varphi I_{s_2})
	\end{bmatrix}.
    \end{equation*}
    
By definition of a geodesic loop, we can write
	
	\begin{equation*}
	\left( P_\varphi^\top \exp_\mathrm{m}\begin{bmatrix}
	2\beta A&-B^\top\\
	B&0
	\end{bmatrix}P_\varphi\right)\left( P_\varphi^\top\exp_\mathrm{m}\begin{bmatrix}
	(1-2\beta)A&0\\
	0&0_p
	\end{bmatrix}P_\varphi\right) P_\varphi^\top I_{2p\times p}= P_\varphi^\top I_{2p\times p}.
	\end{equation*}
    
	Notice that by definition of $P_\varphi$ and by the hypothesis~\eqref{eq:hypot}, we have
	\begin{equation}\label{eq:rightmatrix}
	P_\varphi^\top\exp_\mathrm{m}\begin{bmatrix}
	(1-2\beta)A&0\\
	0&0_p
	\end{bmatrix}P_\varphi = P_\varphi^\top\begin{bmatrix}
	G_p(-\Phi_1)&0\\
	0&I_{p}
	\end{bmatrix}P_\varphi = \begin{bmatrix}
	G_{2s_1}(-\varphi I_{s_1})&0&0\\
	0&I_{2s_2}&0\\
	0&0& *
	\end{bmatrix}.
	\end{equation}
	Moreover, the $2m\times 2m$ upper-left block of $P_\varphi^\top\begin{bmatrix}
	(1-2\beta)A&0\\
	0&0
	\end{bmatrix}P_\varphi$ has also the block structure $ \left[\begin{smallmatrix}
	(1-2\beta) \widehat{A}&0_{s_1\times s_2}\\	
	0_{s_2\times s_1}&0
	\end{smallmatrix}\right]$ where $\widehat{A}$ is the same matrix as previously. By~\eqref{eq:rightmatrix} and \cref{lem:logmatrix}, this yields $\exp_\mathrm{m}((1-2\beta)\widehat{A})=G_{2s_1}(-\varphi I_{s_1})$.
	
	Finally, it follows that
	\begin{equation*}
	\exp_\mathrm{m} \begin{bmatrix}
	2\beta \widehat{A}&-\widehat{B}^\top\\	
	\widehat{B}&0
	\end{bmatrix}=\begin{bmatrix}
	G_{2s_1}(\varphi I_{s_1})&0\\
	0&G_{2s_2}(\varphi I_{s_2})
	\end{bmatrix} \quad\text{and} \quad \exp_\mathrm{m}((1-2\beta)\widehat{A})=G_{2s_1}(-\varphi I_{s_1}). 
	\end{equation*}

\end{proof}

Since $\widehat{A}$ and $\widehat{B}$ are submatrices of $A$ and $B$, respectively, the length of the geodesic loop defined by $\widehat{A}$ and $\widehat{B}$ is smaller than the length of the geodesic loop defined by $A$ and $B$,
\begin{align*}
    \beta\|A\|_\mathrm{F}^2 + \|B\|_\mathrm{F}^2 \geq \beta\|\widehat{A}\|_\mathrm{F}^2 + \|\widehat{B}\|_\mathrm{F}^2.
\end{align*}
Note that by design the matrix $\widehat{A}$ is nonzero.

\begin{remark}\label[remark]{rem:disjunctPhi2_is_zero}
    From equation~\eqref{eq:skewpermuted}, we can deduce that if a diagonal entry $\breve{\varphi}$ of $\Phi_2$ is not in $\Phi_1$, then $\breve{\varphi} = 0$. Equivalently, all nonzero diagonal entries of $\Phi_2$ are in $\Phi_1$.
    Indeed, assume that
    \begin{align*}
        \Phi_2 = \begin{bmatrix}
            \widehat{\Phi}_2 & 0\\
            0 & \Breve{\Phi}_2
        \end{bmatrix},
    \end{align*}
    and assume that the spectrum of $\Breve{\Phi}_2$ does not contain zero and is disjoint from the spectra of $\widehat{\Phi}_2$ and $\Phi_1$. Since $[0,\pi] \ni \varphi \mapsto \{e^{\imath \varphi}, e^{-\imath \varphi}\}$ is injective, it follows that the matrices $\begin{bmatrix}
        G_p(\Phi_1) & 0\\
        0 & G_{q_1}(\widehat{\Phi}_2)
    \end{bmatrix}$ and $G_{q_2}(\Breve{\Phi}_2)$ also have disjoint spectra, with suitable dimensions $q_1,q_2\geq 0$ with $p = q_1+q_2$. Note that $q_1$ is odd, if $p$ is odd.
    By equation~\eqref{eq:skewpermuted} and a similar argument as in the proof of \cref{lem:logmatrix}, we obtain
    \begin{align*}
        \begin{bmatrix}
            2\beta A & -B^\top\\
            B & 0
        \end{bmatrix} \in\exp_\mathrm{m}^{-1}\begin{bmatrix}
            G_p(\Phi_1) & 0\\
            0 & G_p(\Phi_2)
        \end{bmatrix} = \begin{bmatrix}
            \exp_\mathrm{m}^{-1}\begin{bmatrix}
                G_p(\Phi_1) & 0\\
                0 & G_{q_1}(\widehat{\Phi}_2)
            \end{bmatrix} & 0\\
            0 & \exp_\mathrm{m}^{-1}(G_{q_2}(\Breve{\Phi}_2))
        \end{bmatrix}.
    \end{align*}
    So, by identification of the matrix blocks, $0_{q_2\times q_2}\in \exp_\mathrm{m}^{-1}(G_{q_2}(\Breve{\Phi}_2))$ and thus $\Breve{\Phi}_2= 0$. This contradicts the assumption that zero is not in the spectrum of $\Breve{\Phi}_2$. Hence, $\Breve{\Phi}_2$ can not exist.
\end{remark}
\subsection{The \texorpdfstring{$\beta$}{beta}-length of shortest geodesic loops (\texorpdfstring{$\beta\in(2,\infty)$}{beta in (2,oo)})}\label{sec:shortest_loops_g2}

First, we tackle the cases where one of the matrices $A$ or $B$ is zero.
\begin{lemma}\label[lemma]{lem:AorBzero}
    Recall that $\beta >2$ and $2\leq p\leq n-1$. Let $\gamma_\beta$ be a geodesic loop on the Stiefel manifold $\mathrm{St}_\beta(n,p)$ with initial velocity $\Delta =\left[\begin{smallmatrix}
        A\\
        B
    \end{smallmatrix}\right]\in T_{I_{n\times p}}\mathrm{St}_\beta(n,p)$.
    If $B=0$ and $A\neq 0$, then $L_\beta(\gamma_\beta)\geq \sqrt{2\beta}{2\pi}\,(>2\pi)$. Moreover, if $A=0$ and $B\neq 0$, then $L_\beta (\gamma_\beta)\geq 2\pi$. These bounds are sharp.
\end{lemma}
\begin{proof}
    If $A = 0$ (and $B\neq 0$), then $\gamma_\beta$ is also a geodesic loop on $\mathrm{St}_\frac{1}{2}(n, p)$ (canonical metric). Therefore, the length is immediately bounded from below by $2\pi$, cf.~\cite[Thm. 6.1]{absilmataigne2024ultimate}. An example of a geodesic loop of length $2\pi$ is given at the end of this section.

    If $B=0$ and $A\neq 0$, then we have $\gamma_\beta(1) = \left[\begin{smallmatrix}
        \exp_\mathrm{m}(A)\\
        0
    \end{smallmatrix}\right] = \left[\begin{smallmatrix}
        I_p\\
        0
    \end{smallmatrix}\right]$. Therefore, the singular values of $A$ are nonnegative integer multiples of $2\pi$, and, since $A\neq 0$, at least two singular values of $A$ are greater or equal to $2\pi$. Hence, $\|A\|_\mathrm{F}\geq \sqrt{2(2\pi^2)}\geq 2\pi\sqrt{2}$. Finally, $L_\beta(\gamma_\beta) = \sqrt{\beta\|A\|_\mathrm{F}^2}\geq \sqrt{2\beta}2\pi$.
\end{proof}

We now proceed to show, without restriction on $A$ and $B$, that the $\beta$-length of the shortest geodesic loops is bounded below by $2\pi$.
\Cref{sec:structured_loop} exposed the structure of the shortest geodesic loops. In particular, let $\gamma_\beta(t) \coloneq \mathrm{Exp}_{\beta,I_{2p\times p}}\left(t\left[\begin{smallmatrix}
    A\\B
\end{smallmatrix}\right]\right)$ be a geodesic loop which features the structure of~\eqref{eq:schur_geodesic}. If $\Phi_1\neq 0$, by \cref{lem:reduced_geodesic}, we can extract submatrices $\widehat{A}\in\mathrm{Skew}(p_1)$ and $\widehat{B}\in\mathbb{R}^{p_2\times p_1}$ from $A$ and $B$ such that the geodesic loop $\widehat{\gamma}_\beta(t)\coloneq\mathrm{Exp}_{\beta,U}\left(t\left[\begin{smallmatrix}
    \widehat{A}\\
    \widehat{B}
\end{smallmatrix}\right]\right)$ with $U\coloneq I_{p_1+p_2\times p_1}$ features the structure from~\eqref{eq:reduced_loop}. Here, $p_1 = 2s_1$, $p_2 = 2s_2$ for suitable $s_1,s_2\in\mathbb{N}$.
If $\Phi_1 = 0$ is zero, it also holds that the corresponding $\Phi_2 = 0$ is zero (see \cref{rem:disjunctPhi2_is_zero}). Hence, the geodesic loop $\gamma_\beta$ immediately exhibits a structure like \eqref{eq:reduced_loop}
\begin{align*}
    \exp_\mathrm{m}\begin{bmatrix}
	2\beta \widehat{A}&-\widehat{B}^\top\\
	\widehat{B}&0
	\end{bmatrix}=\begin{bmatrix}
	G_{p}(\varphi I_{k})&0\\
	0&G_{p}(\varphi I_{k})
	\end{bmatrix}, \quad \text{and} \quad \exp_\mathrm{m}((1-2\beta)\widehat{A})=G_{p}(-\varphi I_{k}),
\end{align*}
with $\varphi = 0$, $k \coloneq \lfloor\frac{p}{2}\rfloor$ and $\widehat{A} \coloneq A\in\mathrm{Skew}(p_1)$, $\widehat{B} \coloneq B\in\mathbb{R}^{p_2\times p_1}$, with $p_1 = p_2 = p$.

We show that $L_\beta(\widehat{\gamma}_\beta)\geq 2\pi$ and, hence, $L_\beta(\gamma_\beta)\geq 2\pi$. The cases in which either $\widehat{A}$ or $\widehat{B}$ is zero are treated in \cref{lem:AorBzero}. 
If $\widehat{B}$ is the empty matrix ($s_2 = 0$ in \cref{lem:reduced_geodesic}), we can argue analogously to the proof of \cref{lem:AorBzero} to show that the length of the geodesic loop is bounded below by $\sqrt{2\beta}2\pi$. In \cref{lem:length_bound_phi}, we show that $L_\beta(\widehat{\gamma}_\beta) > 2\pi$ if $\widehat{A}\neq0$ and $\widehat{B}\neq 0$. To prove this, we first need to introduce a result that is a direct consequence of the interlacing property of singular values.

\begin{lemma}\label[lemma]{lem:interlacing}
    Recall that $\beta > 2$. Let $p_1,p_2\in\mathbb{N}$ be positive integers and let $A\in\mathrm{Skew}(p_1)$ be a skew-symmetric matrix and $B\in\mathbb{R}^{p_2\times p_1}$ be arbitrary. Define the matrix $X \coloneq \begin{bmatrix}
        2\beta A & -B^\top\\
        B & 0
    \end{bmatrix}$ and $\breve{A} \coloneq (1-2\beta)A$. If we assume that the singular values of $X$ are given by $\sigma_1(X)\geq\dots\geq\sigma_{p_1+p_2}(X)\geq0$ and the singular values of $\breve{A}$ are given by $\sigma_1(\breve{A})\geq\dots\geq\sigma_{p_1}(\breve{A})\geq 0$, then it holds
    \begin{align*}
        \frac{2\beta}{2\beta-1}\sigma_j(\breve{A})\leq\sigma_j(X),\quad\text{for }j = 1,\dots,p_1
    \end{align*}
\end{lemma}
\begin{proof}
    This is a direct consequence of the interlacing property of singular values~\cite[Corollary 7.3.6]{Horn_Johnson_2012}
\end{proof}

Next, we show that the length of the loop $\widehat{\gamma}_\beta$ is also bounded below by $2\pi$ if $\widehat{A}\neq 0$ and $\widehat{B}\neq 0$.

\begin{lemma} \label[lemma]{lem:length_bound_phi}
    Recall that $\beta > 2$. Let $\widehat{A}\in\mathrm{Skew}(p_1)$ and $\widehat{B}\in\mathbb{R}^{p_2\times p_1}$ be the two matrices from above. So, they define the geodesic loop $\widehat{\gamma}_\beta(t)\coloneq\mathrm{Exp}_{\beta,U}\left(t\left[\begin{smallmatrix}
    \widehat{A}\\
    \widehat{B}
    \end{smallmatrix}\right]\right)$, with $U \coloneq I_{p_1+p_2\times p_1}$, that satisfies the structure
    \begin{align}\label{eq:struct_geod_lemmabeta>2}
        \exp_{\mathrm{m}}\begin{bmatrix}
            2\beta\widehat{A} & -\widehat{B}^\top\\
            \widehat{B} & 0
        \end{bmatrix} = \begin{bmatrix}
            G_{p_1}(\varphi I_{s_1}) & 0\\
            0 & G_{p_2}(\varphi I_{s_2})
        \end{bmatrix},\quad\text{and}\quad \exp_{\mathrm{m}}((1-2\beta)\widehat{A}) = G_{p_1}(-\varphi I_{s_1}),
    \end{align}
    where $\varphi\in[0,\pi]$ and $s_1 \coloneq \lfloor\frac{p_1}{2}\rfloor$, $s_2 \coloneq \lfloor\frac{p_2}{2}\rfloor$.
    If we assume that $\widehat{A}$ and $\widehat{B}$ are nonzero, then the length of the geodesic loop $\widehat{\gamma}_\beta$ is greater than $2\pi$, i.e.,
    \begin{align*}
       L_\beta(\widehat{\gamma}_\beta) > 2\pi.
    \end{align*}
\end{lemma}
\begin{proof}
    Recall that $p_1 = p_2 = p$, if $\varphi = 0$, and $p_1 = 2s_1$, $p_2 = 2s_2$, if $\varphi\in(0,\pi]$.
    
    With $X \coloneq \begin{bmatrix}
        2\beta\widehat{A} & - \widehat{B}^\top\\
        \widehat{B} & 0
    \end{bmatrix}$ and $\breve{A} \coloneq (1-2\beta)\widehat{A}$,
    we have 
    $ \|\widehat{B}\|_\mathrm{F}^2 = \frac12\|X\|_\mathrm{F}^2 - 2\beta^2\|\widehat{A}\|_\mathrm{F}^2$ and thus
    \begin{equation}\label{eq:length_from_X}
         L_\beta(\widehat{\gamma}_\beta) \coloneq  \sqrt{\beta\|\widehat{A}\|_\mathrm{F}^2 + \|\widehat{B}\|_\mathrm{F}^2} = \sqrt{\frac12\|X\|_\mathrm{F}^2 - \beta(2\beta-1)\|\widehat{A}\|_\mathrm{F}^2}=\sqrt{\frac12\|X\|_\mathrm{F}^2 - \frac{\beta}{2\beta-1}\|\breve{A}\|_\mathrm{F}^2}.
    \end{equation}
    Equation~\eqref{eq:length_from_X} is a function of the singular values of $X$ and $\breve{A}$. 
    Since $X$ and $\breve{A}$ are skew-symmetric and their matrix exponential satisfies~\eqref{eq:struct_geod_lemmabeta>2}, the singular values of $X$ come in pairs 
    \begin{align}\label{eq:singvals_X}
    |\varphi + 2\pi k_j|,|\varphi + 2\pi k_j|,\text{ for } k_j\in\mathbb{Z} \text{ and } j\in\{1,\dots,p_1+p_2\},
    \end{align}
    and so do the singular value of $\breve{A}$, 
    \begin{align}\label{eq:singvals_A}
    |\varphi + 2\pi\ell_j|,|\varphi+2\pi\ell_j|,\text{ for } \ell_j\in\mathbb{Z} \text{ and } j\in\{1,\dots,p_1\}.
    \end{align}
    If $p_1$ is odd, then $\breve{A}$ additionally has the singular value zero.
    We assume, w.l.o.g., that the singular value (pairs) are ordered by magnitude.  
    By \cref{lem:interlacing}, we know that if we compare the $s_1$ largest singular value pairs of $X$ with those of $\breve{A}$ in descending order, the singular value pair of $X$ is greater in magnitude than the singular value pair of $\breve{A}$ in each pairing. We show that there are pairings in which the difference between the values of the singular value pairs of $X$ and $\breve{A}$ is greater than $2\pi$.
    We start by showing this property for $\varphi = \pi$. Let $j_a\in\{1,\dots,s_1\}$ be arbitrary. By \cref{lem:interlacing}, 
    \begin{align*}
        0<|\pi + 2\pi\ell_{j_a}|<\frac{2\beta}{2\beta-1}|\pi + 2\pi \ell_{j_a}|\leq |\pi + 2\pi k_{j_a}|.
    \end{align*}
    For the integers $\ell_{j_a}$, $k_{j_a}$, we therefore have $|2\ell_{j_a}+1|<|2k_{j_a}+1|$ and thus
    \begin{align*}
        |\pi + 2\pi k_{j_a}| - |\pi + 2\pi\ell_{j_a}|\geq 2\pi.
    \end{align*}
    Next, let $\varphi = 0$. Since $\widehat{A}\neq 0$ is nonzero, there is a $j_a\in\{1,\dots,s_1\}$ such that the corresponding singular value pair of $\breve{A}$ is nonzero, $|2\pi\ell_{j_a}| > 0$. By \cref{lem:interlacing},
    \begin{align*}
        0 < |2\pi\ell_{j_a}| < \frac{2\beta}{2\beta-1}|2\pi\ell_{j_a}| \leq |2\pi k_{j_a}|.
    \end{align*}
    For the integers $\ell_{j_a}$ and $k_{j_a}$, we therefore have $|\ell_{j_a}| < |k_{j_a}|$, and thus
    \begin{align*}
        |2\pi k_{j_a}| - |2\pi \ell_{j_a}| \geq 2\pi.
    \end{align*}
    For $\varphi\in(0,\pi)$, there is also a $j_a\in\{1,\dots,s_1\}$ such that 
    \begin{align} \label{eq:sing_val_property_phi0pi}
        |\varphi + 2\pi k_{j_a}| - |\varphi + 2\pi\ell_{j_a}|\geq 2\pi.
    \end{align}
    To see this, we refer the reader to \cref{lem:singval_property} in Appendix \ref{app:singval_property}.
    With this property of the singular values of $X$ and $\breve{A}$, we show that the length of the geodesic loop $\widehat{\gamma}_\beta$ is greater than $2\pi$. 
    For any $\varphi\in[0,\pi]$, there exists a $j_a\in\{1,\dots,s_1\}$ corresponding to the $j_a$-th largest singular value pair of $X$ and $\breve{A}$ such that 
    \begin{align}\label{eq:singvalProperty}
        |\varphi + 2\pi k_{j_a}| - |\varphi + 2\pi\ell_{j_a}| \geq 2\pi.
    \end{align}
    By \cref{lem:interlacing},
    \begin{align}\label{eq:interlacing_for_j_a}
        \frac{2\beta}{2\beta-1}|\varphi + 2\pi\ell_{j_a}| \leq |\varphi + 2\pi k_{j_a}|.
    \end{align}
    So, by squaring and rearranging equation~\eqref{eq:singvalProperty}, we can make use of equation~\eqref{eq:interlacing_for_j_a} to obtain
    \begin{align*}
        |\varphi + 2\pi k_{j_a}|^2 &\geq 4\pi^2 + 2|\varphi + 2\pi k_{j_a}||\varphi + 2\pi\ell_{j_a}| - |\varphi + 2\pi\ell_{j_a}|^2\\
        &\geq 4\pi^2 + \left(2\frac{2\beta}{2\beta-1} - 1\right)|\varphi + 2\pi\ell_{j_a}|^2.
    \end{align*}
    It also follows from equation~\eqref{eq:interlacing_for_j_a} that
    \begin{align}\label{eq:squared_interlacing_connection}
        \frac{2\beta}{2\beta-1}|\varphi + 2\pi\ell_j|^2\leq\left(\frac{2\beta}{2\beta-1}|\varphi + 2\pi\ell_j|\right)^2\leq |\varphi + 2\pi k_j|^2,
    \end{align}
    for all $j\in\{1,\dots,s_1\}$. Hence, it follows by~\eqref{eq:length_from_X} that
    \begin{align*}
         L_\beta(\widehat{\gamma}_\beta)^2 &= \frac12\|X\|_\mathrm{F}^2 - \frac{\beta}{2\beta-1}\|\breve{A}\|_\mathrm{F}^2\\
        & = \sum_{j=1}^{s_1}\underbrace{\left(|\varphi + 2\pi k_j|^2 - \frac{2\beta}{2\beta-1}|\varphi +2\pi\ell_j|^2\right)}_{\geq 0\text{, see eq.~\eqref{eq:squared_interlacing_connection}}} + \sum_{j=s_1 + 1}^{\frac{p_1+p_2}{2}} \underbrace{|\varphi + 2\pi k_j|^2}_{\geq 0}\\
        &\geq |\varphi + 2\pi k_{j_a}|^2 - \frac{2\beta}{2\beta-1}|\varphi + 2\pi\ell_{j_a}|^2\\
        &\geq 4\pi^2 + \left(\frac{2\beta}{2\beta - 1} - 1\right)\underbrace{|\varphi + 2\pi\ell_{j_a}|^2}_{>0}\\
        &> 4\pi^2.
    \end{align*}
    The length of the geodesic loop $\widehat{\gamma}_\beta$ is thus greater than $2\pi$. This concludes the proof.
\end{proof}

In total, we proved the following:
\begin{theorem}\label{thm:closedgeod_beta>2}
    Under the standing assumptions that $\beta > 2$ and $2\leq p\leq n-1$, the length $\ell_\beta$ of a shortest geodesic loop on the Stiefel manifold $\mathrm{St}_\beta(n,p)$ is $\ell_\beta = 2\pi$.
\end{theorem}
\begin{proof}
    Above, we prove that the length of a geodesic loop on the Stiefel manifold $\mathrm{St}_\beta(n,p)$ is at least $2\pi$. From~\cite[Section 6]{absilmataigne2024ultimate}, we obtain that there is a geodesic loop on $\mathrm{St}_\beta(n,p)$ with length $2\pi$.
\end{proof}

An example of a geodesic loop of length $2\pi$ is  $\gamma_\beta(t)\coloneq\mathrm{Exp}_{\beta,I_{3\times 2}}\left(t\left[\begin{smallmatrix}A\\B\end{smallmatrix}\right]\right)$ with $A = 0\in\mathrm{Skew}(2)$ and $B = \begin{bmatrix}
    2\pi & 0
\end{bmatrix}$. It holds
\begin{align*}
    \gamma_\beta(1) &= \exp_{\mathrm{m}}\begin{bmatrix}
        2\beta A & -B^\top\\
        B & 0
    \end{bmatrix}I_{3\times 2}\exp_{\mathrm{m}}((1-2\beta)A)\\
    & = I_3I_{3\times 2}I_2 = I_{3\times 2} = \gamma_\beta(0)
\end{align*}
and $ L_\beta(\gamma_\beta) = \sqrt{\beta\|A\|_\mathrm{F}^2 + \|B\|_\mathrm{F}^2} = \|B\|_\mathrm{F}^2 = 2\pi$. This geodesic loop can be embedded in any Stiefel manifold $\mathrm{St}_\beta(n,p)$ with $2\leq p\leq n-1$ by padding the matrices $A$ and $B$ with zeros.

\begin{corollary}
    Let $\gamma_\beta(t)\coloneq \mathrm{Exp}_{\beta,I_{n\times p}}\left(t\left[\begin{smallmatrix}A\\B\end{smallmatrix}\right]\right)$ be a geodesic loop on the Stiefel manifold $\mathrm{St}_\beta(n,p)$ with length $2\pi$, i.e., $L_\beta(\gamma_\beta) = 2\pi$. Then, it holds that $A = 0$.
\end{corollary}

\section{Bounds for the sectional curvature}\label{sec:bounds_sectcurvature}
As outlined in equations~\eqref{eq:inj=min(l/2,conj)} and~\eqref{eq:conjbound} from \Cref{sec:geod_and_injrad_intro}, the injectivity radius is given by $\mathrm{inj}(\mathrm{St}_\beta(n,p)) = \min\left\{\frac{\ell_\beta}{2},\mathrm{conj}_{I_{n\times p}}(\mathrm{St}_\beta(n,p))\right\}$ and the conjugate radius is bounded by $\mathrm{conj}_{I_{n\times p}}(\mathrm{St}_\beta(n,p)) \geq \frac{\pi}{\sqrt{K_\beta}}$. In \Cref{sec:betalength<=2,sec:closedgeod_beta>2}, we have calculated the length $\ell_\beta$ of a shortest geodesic loop for all values of $\beta>0$. Next, we provide a positive constant $K_\beta$ that globally bounds the sectional curvature of $\mathrm{St}_\beta(n,p)$ from above in order to obtain a bound for the conjugate radius. This allows us to make a statement about the injectivity radius in \Cref{sec:injectivity}. 
We exclude the cases $n=2, p=2$ and $n=2, p=1$, since here the Stiefel manifold is one-dimensional, so that the concept of sectional curvature does not apply. From~\cite[Theorem 12, Theorem 13, Theorem 14]{zimmermannstoye_curvature:2024}\footnote{In~\cite[Theorem 12]{zimmermannstoye_curvature:2024}, it is stated that the sectional curvature is bounded below by zero for all $\beta\in[\frac16(4-\sqrt{10}),\frac23]$. However, the proof for this statement is missing, and in fact, negative curvature is possible if $\beta \in (\frac{1}{2}, \frac{2}{3}]$; see \cite[Table~2 and the definition of $\mathfrak{l}(\alpha)$]{nguyen2022curvature}.},
\begin{eqnarray*}
    K_{\beta} = \begin{cases}
        \frac{1}{4\beta} & 0<\beta\leq\frac16(4-\sqrt{10})\approx 0.14...,\\
        \frac{4-3\beta}{2} & \frac16(4-\sqrt{10})<\beta\leq\frac23,\\
        1 & 0.71...\approx\frac{1}{\sqrt{2}}\leq\beta\leq 1.
    \end{cases}
\end{eqnarray*}
There is a gap for $\beta\in(\frac23,\frac{1}{\sqrt{2}})$ which is not covered in~\cite{zimmermannstoye_curvature:2024}. In this section, we close this gap. 

A formula for the sectional curvature is derived in~\cite[eq. (34)]{nguyen2022curvature}.
For a $\beta$-orthonormal basis $\{\Delta_1,\Delta_2\} \in T_{I_{n\times p}}\mathrm{St}_\beta(n,p)$ of a two-dimensional tangent space section, it reads
\begin{align}\label{eq:sect_curv_formula}
\begin{split}
	\mathcal{K}_\beta(\Delta_1,\Delta_2) 
	=&  \frac12\|B_1B_2^\top - B_2B_1^\top\|_\mathrm{F}^2 + 
 \frac{(1-2\beta)^3}{2}\|B_1^\top B_2 - B_2^\top B_1\|_\mathrm{F}^2
 \\ 
	&  +\beta^2\|B_1A_2 - B_2A_1\|_\mathrm{F}^2
    \\ &  + \frac{\beta}{4} \| [A_1,A_2] - (3-4\beta)(B_1^\top B_2 - B_2^\top B_1)\|_\mathrm{F}^2.
\end{split}
\end{align}
The matrix blocks that form $\Delta_1,\Delta_2$ are related by 
$\|\Delta_i\|_\beta^2 = \beta \|A_i\|_\mathrm{F}^2 + \|B_i\|_\mathrm{F}^2 =1$ for $i=1,2$ and
$\langle \Delta_1,\Delta_2\rangle_\beta =\beta \mathrm{tr}(A_1^\top A_2) + \mathrm{tr}(B_1^\top B_2)=0$.

As in~\cite{zimmermannstoye_curvature:2024}, we have the following term-by-term estimates, where $\alpha_i \coloneq \|A_i\|_\mathrm{F}$, $\eta_i\coloneq\|B_i\|_\mathrm{F}$ and $\omega \coloneq \|B_1^\top B_2 - B_2^\top B_1\|_\mathrm{F}$:
\begin{itemize}
    \item $\frac12\|B_1B_2^\top - B_2B_1^\top\|_\mathrm{F}^2\leq \eta_1^2\eta_2^2$,

    \item $\beta^2\|B_1A_2 - B_2A_1\|_\mathrm{F}^2\leq \frac{\beta^2}2\eta_1^2\alpha_2^2 + \frac{\beta^2}2 \eta_2^2 \alpha_1^2 + \beta^2\alpha_1\alpha_2\eta_1\eta_2$,

    \item $\| [A_1,A_2]\|_\mathrm{F}^2\leq \alpha_1^2\alpha_2^2$,

    \item $\| [A_1,A_2] - (3-4\beta)(B_1^\top B_2 - B_2^\top B_1)\|_\mathrm{F}^2\leq\alpha_1^2\alpha_2^2 + 2(3-4\beta)\alpha_1\alpha_2\omega + (3-4\beta)^2\omega^2$.
\end{itemize}
We start by using the term-by-term estimates and the relation $\eta_i^2 = 1-\beta\alpha_i^2$ to obtain 
\begin{align}
    \mathcal{K}_\beta(\Delta_1,\Delta_2) \leq&
    \eta_1^2\eta_2^2 + \frac{(1-2\beta)^3}{2}\omega^2\nonumber\\
    &  + \left(-\beta^3\alpha_1^2\alpha_2^2 + \frac{\beta^2}{2}\left(\alpha_1^2 + \alpha_2^2\right) + \beta^2\alpha_1\alpha_2\eta_1\eta_2\right)\nonumber\\
    &  + \frac{\beta}{4}\alpha_1^2\alpha_2^2 + \frac{\beta}{4}(3-4\beta)^2\omega^2 + \frac{\beta}{4}2(3-4\beta)\alpha_1\alpha_2\omega,\nonumber\\
    \begin{split}\label{eq:upperbound}
    = & \frac{2-3\beta}{4}\omega^2 + \frac{2\beta(3-4\beta)}{4}\alpha_1\alpha_2\omega + \eta_1^2\eta_2^2\\ 
    & +\left(\frac{\beta}{4}-\beta^3\right)\alpha_1^2\alpha_2^2 + \frac{\beta^2}{2}\left(\alpha_1^2+\alpha_2^2\right) + \beta^2\alpha_1\alpha_2\eta_1\eta_2.
\end{split}
\end{align}
This is a parabola in $\omega$ with a downward opening for any value of $\beta\in(\frac23,\frac{1}{\sqrt{2}})$ and, by the matrix inequality of Wu and Chen~\cite{WuChen1988}, we have
\begin{align*}
    \omega = \|B_1^\top B_2 - B_2^\top B_1\|_\mathrm{F} \leq \sqrt{2}\|B_1\|_\mathrm{F}\|B_2\|_\mathrm{F} \leq \sqrt{2}.
\end{align*}

We show that the maximal value of the parabola is bounded above by one and, hence, the sectional curvature is bounded by one, thereby closing the aforementioned gap. 
The maximum value of the parabola is attained at $\omega_{\max}\coloneq\frac{\beta(3-4\beta)}{(3\beta-2)}\alpha_1\alpha_2$, if $\omega_{\max} < \sqrt{2}$, or at $\omega =\sqrt{2}$, if $\omega_{\max}\geq\sqrt{2}$. 
Assume that $\omega = \sqrt{2}$. Then, we necessarily have $\|B_1\|_\mathrm{F} = \|B_2\|_\mathrm{F} = 1$ and therefore $\|A_1\|_\mathrm{F} = \|A_2\|_\mathrm{F} = 0$. In this case, the expression for the upper bound~\eqref{eq:upperbound} becomes
\begin{align*}
   ~\eqref{eq:upperbound}|_{\omega = \sqrt{2}} = \frac{2-3\beta}{4}\sqrt{2}^2 + 1 = \frac{4-3\beta}{2}(<1).
\end{align*}
Now, assume that $\omega_{\max} < \sqrt{2}$. Plugging in $\omega_{\max}$ into the expression for the upper bound~\eqref{eq:upperbound} gives
\begin{align}\label{eq:upperbound_omegamax}
   ~\eqref{eq:upperbound}|_{\omega = \omega_{\max}} = \underbrace{\left(\frac{\beta^2(3-4\beta)^2}{4(3\beta-2)} - \beta^3 + \frac{\beta}{4}\right)}_{\eqcolon T_\beta}\alpha_1^2\alpha_2^2 + \frac{\beta^2}{2}(\alpha_1^2 + \alpha_2^2) + \beta^2\alpha_1\alpha_2\eta_1\eta_2 + \eta_1^2\eta_2^2.
\end{align}
Next, we show that~\eqref{eq:upperbound_omegamax} is bounded above by one. Recall that $t_1t_2\leq \frac12(t_1^2+t_2^2)$ for any $t_1,t_2\in\mathbb{R}$. By this fact and the relation $\eta_i^2 = 1 - \beta\alpha_i^2$, we have:
\begin{itemize}
    \item $\eta_1\eta_2\leq\frac12(\eta_1^2+\eta_2^2) = 1 - \frac{\beta}{2}(\alpha_1^2 + \alpha_2^2) \leq 1-\beta\alpha_1\alpha_2$,
    \item $\left(\frac{\beta^2}{2} - \beta\right)(\alpha_1^2+\alpha_2^2)\leq (\beta^2-2\beta)\alpha_1\alpha_2$, with $\beta\in(\frac23,\frac{1}{\sqrt{2}})$.
\end{itemize}
So,
\begin{align*}
   ~\eqref{eq:upperbound_omegamax} &\leq T_\beta\alpha_1^2\alpha_2^2 + \frac{\beta^2}{2}(\alpha_1^2+\alpha_2^2) + \beta^2\alpha_1\alpha_2(1-\beta\alpha_1\alpha_2) + (1-\beta\alpha_1^2)(1-\beta\alpha_2^2)\\
    &=(T_\beta + \beta^2 - \beta^3)\alpha_1^2\alpha_2^2 + \left(\frac{\beta^2}{2}-\beta\right)(\alpha_1^2+\alpha_2^2) + \beta^2\alpha_1\alpha_2 + 1\\
    &\leq (T_\beta + \beta^2 - \beta^3)\alpha_1^2\alpha_2^2 + 2(\beta^2 - \beta)\alpha_1\alpha_2 + 1.
\end{align*}
This is a parabola in $\alpha_1\alpha_2$ with an upward opening, since $(T_\beta +\beta^2-\beta^3)\geq 0$ holds for $\beta\in(\frac23,\frac{1}{\sqrt{2}})$. Therefore, the parabola becomes maximal for $\alpha_1\alpha_2$ as large or as small as possible. In our case, we have $\alpha_1\alpha_2\geq 0$ and, since $\omega_{\max} < \sqrt{2}$, we have $\alpha_1\alpha_2 < \frac{\sqrt{2}}{\beta}\frac{3\beta-2}{3-4\beta}$. The value of the parabola at $\alpha_1\alpha_2 = 0$ is one. For $\alpha_1\alpha_2 = \frac{\sqrt{2}}{\beta}\frac{3\beta-2}{3-4\beta}$, we show that the value of the parabola is smaller than one. With $\beta\in(\frac{2}{3},\frac{1}{\sqrt{2}})$, it can be shown that this is equivalent to 
\begin{align*}
   0 < (4+8\sqrt{2})\beta^3 - 2(1+7\sqrt{2})\beta^2+(6\sqrt{2}-2)\beta + 1.
\end{align*}
It is easy to verify that the cubic polynomial on the right-hand side is positive for any $\beta\in(\frac{2}{3},\frac{1}{\sqrt{2}})$. 
Hence, the gap in~\cite[Theorem 12, Theorem 13, Theorem 14]{zimmermannstoye_curvature:2024} is filled and combined, we have

\begin{theorem}\label{thm:sect_curvature}
    The sectional curvature of the Stiefel manifold $\mathrm{St}_\beta(n,p)$ is bounded above by $K_\beta>0$ with
    \begin{align*}
        K_\beta = \begin{cases}
            \frac{1}{4\beta} & 0<\beta\leq\frac16(4-\sqrt{10}),\\
            \frac{4-3\beta}{2} & \frac16(4-\sqrt{10})<\beta\leq\frac23,\\
            1 & \frac23\leq\beta\leq 1.
        \end{cases}
    \end{align*}
\end{theorem}
For a graphical illustration, see \cref{fig:K_beta}.
\begin{figure}[ht]
    \centering
    \includegraphics[width=0.5\linewidth]{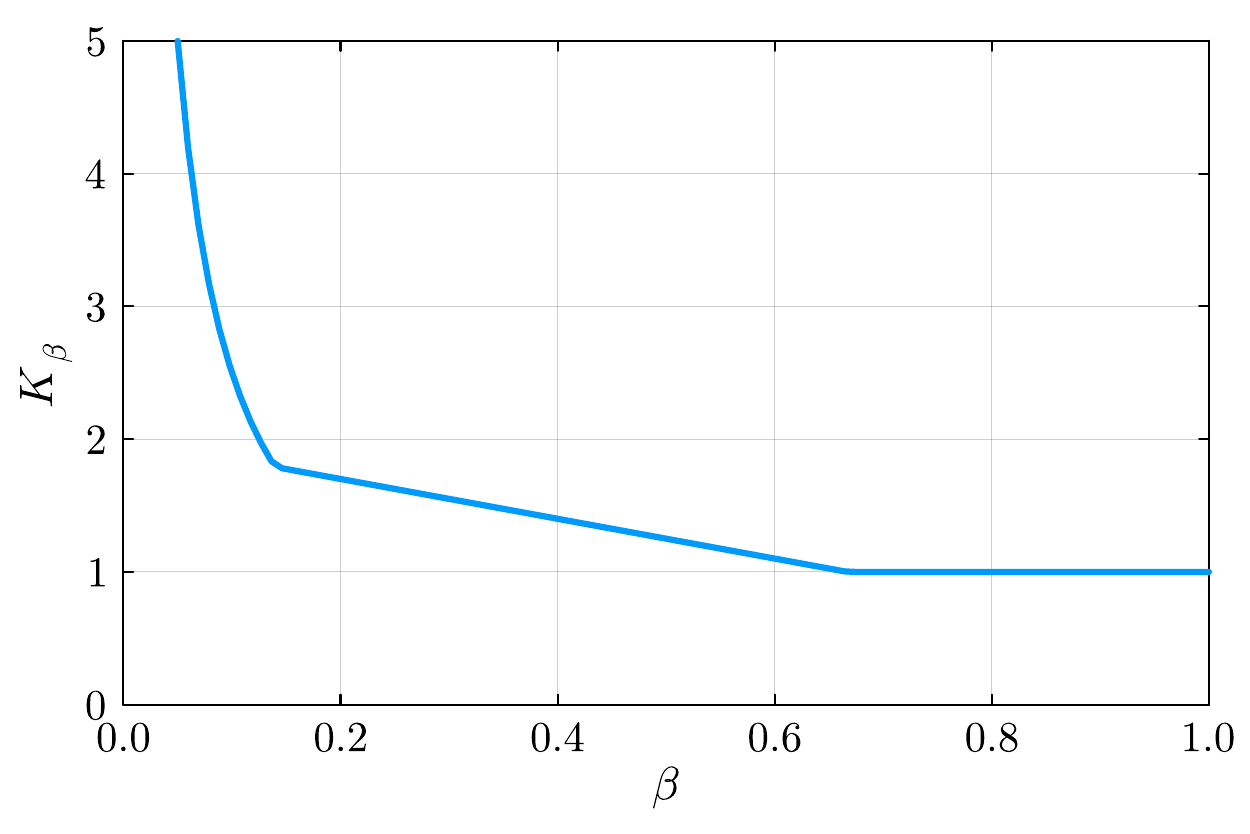}
    \caption{The upper bound $K_\beta$ on the sectional curvature of $\mathrm{St}_\beta(n,p)$ for $\beta\in(0,1]$.}
    \label{fig:K_beta}
\end{figure}
\begin{remark}
    For $2\leq p\leq n-2$, the bound $K_\beta$ on the sectional curvature of $\mathrm{St}_\beta(n,p)$ for $\beta\in(\frac23,\frac{1}{\sqrt{2}})$ is sharp and is only attained for tangent sections with zero skew-symmetric $A$-blocks and $B_1$, $B_2$ such that $\Vert B_1^\top B_2 - B_2^\top B_1\Vert_\mathrm{F} = 0$ and $\Vert B_1B_2^\top - B_2B_1^\top\Vert_\mathrm{F} = \sqrt{2}$. e.g., $A_1 = A_2 = 0$ and
    \begin{align*}
        B_1 = \begin{bmatrix}
            1 & 0\\
            0 & 0
        \end{bmatrix},\,B_2 = \begin{bmatrix}
            0 & 0\\
            1 & 0
        \end{bmatrix}.
    \end{align*}
    Of course, the curvature is the same for all matrices of higher dimensions that feature the above examples as subblocks and are otherwise filled up with zeros.
    
    Together with~\cite[Theorem 12, Theorem 13, Theorem 14]{zimmermannstoye_curvature:2024}, we obtain that the curvature bounds given in \cref{thm:sect_curvature} are sharp for $4\leq p\leq n-2$. The same bounds were previously conjectured in ~\cite[Table 2]{nguyen2022curvature}.
\end{remark}

\section{Implications for the injectivity radius}\label{sec:injectivity}

Recall the results due to Klingenberg~\cite{klingenberg1982}, displayed in equations~\eqref{eq:inj=min(l/2,conj)} and~\eqref{eq:conjbound} in \Cref{sec:geod_and_injrad_intro}: The injectivity radius is given by
$$\mathrm{inj}(\mathrm{St}_\beta(n,p)) = \min\left\{\frac{\ell_\beta}{2}, \mathrm{conj}_{I_{n\times p}}(\mathrm{St}_\beta(n,p))\right\}$$
and the conjugate radius is bounded from below by 
$$\mathrm{conj}_{I_{n\times p}}(\mathrm{St}_\beta(n,p))\geq \frac{\pi}{\sqrt{K_\beta}}.$$
We can now exploit the knowledge on the length $\ell_\beta$ of the shortest geodesic loops and the bound $K_\beta$ on the sectional curvature gathered in the previous sections in order to determine the exact value for, or an upper bound on, the injectivity radius of the Stiefel manifold $\mathrm{St}_\beta(n,p)$.

\subsection{Case $p=1$ or $p\geq n-1$}

As previously noted, when $p=1$, the Stiefel manifold reduces to the unit sphere as a Riemannian submanifold of the Euclidean space $\mathbb{R}^n$, independently of the parameter $\beta$. The injectivity radius is $\mathrm{inj}(\mathrm{St}_\beta(n,1)) = \pi$; see, e.g.,~\cite[\S 1.6]{cheeger1975comparison}. In this case, half the length of the shortest geodesic loops coincides with the conjugate radius. Since the geodesics on the sphere are great circles, the shortest geodesic loops have length $2\pi$, and thus $\frac{\ell_\beta}{2} = \pi$. Moreover, the unit sphere has constant sectional curvature equal to one, implying that the conjugate radius is bounded by $\pi$. By~\cite[Chapter 5 Example 3.3]{docarmo} this bound is attained, so the conjugate radius is exactly $\pi$.

When $p=n\geq 2$, the Stiefel manifold reduces to the orthogonal group $\mathrm{O}(n)$ endowed with the Frobenius metric scaled by $\beta$. From~\cite[Section 6]{absilmataigne2024ultimate}, we obtain $\ell_\beta = \sqrt{2\beta}2\pi$. This result can alternatively be derived using an argument analogous to that in the proof of \cref{lem:AorBzero}. 
In the next theorem, we determine the exact value of the conjugate radius of $\mathrm{St}_\beta(n,n)$, for $n > 2$.
When $n=2$, the manifold is one-dimensional, and therefore no conjugate points exist and the injectivity radius is immediately given by $\mathrm{inj}(\mathrm{St}_\beta(2,2)) = \sqrt{2\beta}\pi$.

\begin{theorem}
Let $n > 2$. The conjugate radius of the orthogonal group $\mathrm{O}(n)$ endowed with the Frobenius metric scaled by $\beta$ is given by 
\[\mathrm{conj}_{I_n}(\mathrm{St}_\beta(n,n)) =\begin{cases} \sqrt{8\beta}\pi & n=3,\\
\sqrt{4\beta}\pi & n>3.
\end{cases}\]
Therefore, the injectivity radius is given by $\mathrm{inj}(\mathrm{St}_\beta(n,n))= \sqrt{2\beta}\pi$.
\end{theorem}
\begin{proof}
    We begin by establishing a lower bound on the conjugate radius using equation~\eqref{eq:conjbound}.
    For $p=n>2$, the sectional curvature formula~\eqref{eq:sect_curv_formula} simplifies to $\mathcal{K}_\beta(\Delta_1,\Delta_2) = \frac{\beta}{4}\|[A_1,A_2]\|_F^2$, where $A_1,A_2\in\mathrm{Skew}(n)$ satisfy $\beta\|A_i\|_F^2\leq 1$ for $i = 1,2$. Invoking the inequality $\|[A_1,A_2]\|_F^2\leq\|A_1\|_F^2\|A_2\|_F^2$ from~\cite[Lemma 2.5]{Ge2014}, we obtain $\mathcal{K}_\beta(\Delta_1,\Delta_2) \leq \frac{1}{4\beta} \eqcolon K_\beta$.
    Consequently, the conjugate radius is bounded below by 
    \[\mathrm{conj}_{I_n}(\mathrm{St}_\beta(n,n))\geq\sqrt{4\beta}\pi.\]
    Next, we describe the relationship between conjugate points and critical points of the matrix exponential. The tangent space at $I_n$, i.e., the Lie algebra of $\mathrm{O}(n)$, is given by the set of skew-symmetric matrices $T_{I_n}\mathrm{St}(n,n) = \mathrm{Skew}(n)$, and the geodesic starting in $I_n$ with initial velocity $\widetilde{A}\in\mathrm{Skew}(n)$ is $\gamma(t) = \exp_\mathrm{m}(t\widetilde{A})$. Hence, the Riemannian exponential at $I_n$ takes the form \[\mathrm{Exp}_{\beta,I_n}\colon\mathrm{Skew}(n)\to\mathrm{St}_\beta(n,n),\quad\mathrm{Exp}_{\beta,I_n}(\widetilde{A}) = \exp_\mathrm{m}(\widetilde{A}).\] 
    According to~\cite[Chapter~5 Proposition~3.5]{docarmo}, a point $\gamma(t_0)$ is conjugate to $\gamma(0) = I_n$ along $\gamma$ if and only if the differential $\mathrm{D}\exp_\mathrm{m}(A)$ is rank-deficient, where $A \coloneq t_0\widetilde{A}$. The corresponding geodesic distance is $\|A\|_\beta = \sqrt{\beta\|A\|_\mathrm{F}^2}$.
    In~\cite[Definition~4.1]{deng2025exponentialskewsymmetricmatricesnearby} the tangent conjugate locus of $\mathrm{O}(n)$ is described. As a direct consequence, we obtain a specific condition on the eigenvalues of $A$ for $\mathrm{D}\exp_\mathrm{m}(A)$ to be rank-deficient.
    Let the eigenvalues of $A$ be given by $i\theta_1,-i\theta_1,\dots,i\theta_m,-i\theta_m$, for $n = 2m$; when $n = 2m+1$, there is an additional zero eigenvalue.
    For $n$ even, $\mathrm{D}\exp_\mathrm{m}(A)$ is rank-deficient if and only if there exist $i\neq j$ and a nonzero integer $\ell$ such that $\theta_i\pm\theta_j = 2\pi\ell$. When $n$ is odd, the same condition applies, or alternatively, $\theta_i = 2\pi\ell$ for some $i$ and $0\neq\ell\in\mathbb{Z}$. 
    For $n>3$, consider the skew-symmetric matrix $A$ whose upper-left $4\times 4$ block is $I_2\otimes \begin{bmatrix}0 & -\pi\\\pi&0\end{bmatrix}$ with all other entries zeros. This matrix satisfies the above criterion, since $\theta_1 = \theta_2 = \pi$ and $\theta_1+\theta_2 = 2\pi$. Therefore, $A$ corresponds to a conjugate point at a geodesic distance of $\sqrt{\beta\|A\|_\mathrm{F}^2} = \sqrt{4\beta}\pi$. Hence, the lower bound on the conjugate radius is sharp and $\mathrm{conj}(\mathrm{St}_\beta(n,n)) = \sqrt{4\beta}\pi$, for $n>3$.
    For $n=3$, the eigenvalues of $A$ are $0$ and $\pm i\theta_1$. For $A$ to define a conjugate point, it has to hold $\theta_1 = 2\pi\ell$ with $0\neq\ell\in\mathbb{Z}$. The corresponding geodesic distance of the conjugate point is $\sqrt{\beta\|A\|_\mathrm{F}^2} = \sqrt{2\beta\theta_1^2} = \sqrt{8\beta}\pi|\ell|$. The minimum occurs for $\ell = 1$, for instance with $A = \begin{bmatrix}0&-2\pi&0 \\ 2\pi&0&0 \\ 0&0&0\end{bmatrix}$. Hence, the conjugate radius is given by $\mathrm{conj}(\mathrm{St}_\beta(3,3)) = \sqrt{8\beta}\pi$. 

    Since half the length of the shortest geodesic loops is smaller than the conjugate radius in either case, we obtain by equation~\eqref{eq:inj=min(l/2,conj)} that the injectivity radius is given by $\mathrm{inj}(\mathrm{St}_\beta(n,n)) = \frac{\ell_\beta}{2} = \sqrt{2\beta}\pi$.
\end{proof}

It is worth noting that the choice $\beta = 1$ yields the Frobenius inner product, see, e.g.,~\cite[Example~3.5]{guigui2023}. For $\beta = \frac12$, the shortest geodesic loops have the nice-looking length $2\pi$; a choice made, e.g., in~\cite[(3.12)]{absilmataigne2024ultimate}.

When $p=n-1$, we are able to determine the exact value of the injectivity radius up to $\beta = 1$.

\begin{theorem}\label{thm:injrad_p=n-1}
    For the injectivity radius of the Stiefel manifold $\mathrm{St}_\beta(n,n-1)$, it holds
    \begin{align*}
        \mathrm{inj}(\mathrm{St}_\beta(n,n-1))
        \begin{cases}
            = \sqrt{2\beta}\pi & 0<\beta<\frac12,\\
            = \pi & \frac12\leq\beta\leq 1,\\
            \leq \pi & 1 < \beta.
        \end{cases}
    \end{align*}
\end{theorem}
\begin{proof}
    For $\beta > 1$, the length of the shortest geodesic loops is given by $\ell_\beta = 2\pi$ in \cref{thm:shortest_beta_geos<=2,thm:closedgeod_beta>2}. Hence, by equation~\eqref{eq:inj=min(l/2,conj)}, we immediately obtain that the injectivity radius is bounded by $\pi$, i.e., $\mathrm{inj}(\mathrm{St}_\beta(n,n-1))\leq \pi$.

    For $\beta \leq 1$, in addition to the bounds on the length of the shortest geodesic loops (\cref{thm:shortest_beta_geos<=2}), we also have bounds on the sectional curvature (\cref{thm:sect_curvature}), which in turn provide bounds on the conjugate radius. From equations~\eqref{eq:inj=min(l/2,conj)} and~\eqref{eq:conjbound}, it follows that the injectivity radius of $\mathrm{St}_{\beta}(n,n-1)$ equals half the length of a shortest geodesic loop whenever
    \begin{align}\label{eq:closedgeodshortercurv_p=n-1}
        \frac{\ell_\beta}{2} \leq \frac{\pi}{\sqrt{K_\beta}}.
    \end{align} 
    It is straightforward to show that this inequality holds for $\beta\in(0,\frac13]\cup[\frac23,1]$. In Appendix \ref{sec:injradp=n-1}, we derive a sharper bound on the sectional curvature of $\mathrm{St}_\beta(n,n-1)$ for $\beta\in(\frac13,\frac23)$, given by $K_\beta = \frac{1}{2\beta}$. With this refined bound, inequality~\eqref{eq:closedgeodshortercurv_p=n-1} is also satisfied for $\beta\in(\frac13,\frac23)$. Consequently, the injectivity radius of $\mathrm{St}_\beta(n,n-1)$ is given by half the length of a shortest geodesic loop, $\frac{\ell_\beta}{2}$, for all $\beta\in[0,1]$.
\end{proof}

\subsection{Case $2\leq p\leq n-2$}

\begin{theorem}\label{thm:injectivity_radius}
    Let $2\leq p \leq n-2$ and consider the Stiefel manifold $\mathrm{St}_\beta(n,p)$. Let $K_\beta$ be the upper bound on the sectional curvature from \cref{thm:sect_curvature} and $\ell_\beta$ be the length of the shortest geodesic loops from \cref{thm:shortest_beta_geos<=2,thm:closedgeod_beta>2}. Define $t_\beta^\mathrm{r}\coloneq \min\{t > 0\ | \ \frac{\sin t}{t}+\frac{1-\beta}{\beta}\cos t=0\}$.
    Then, we obtain
    \[
        \mathrm{inj}(\mathrm{St}_\beta(n,p))
        \begin{cases}
            = \sqrt{2\beta}\pi & 0<\beta\leq\frac13,\\
            \in [ \frac{\pi}{\sqrt{K_\beta}}, \min\{\frac{\ell_\beta}{2},\sqrt{2}t_\beta^\mathrm{r}\} ] & \frac13<\beta<\frac23,\\
            = \pi & \frac23\leq\beta\leq 1,\\
            \leq \pi & 1 < \beta.
        \end{cases}
    \]
\end{theorem}
\begin{proof}
    Analogously to the proof of \cref{thm:injrad_p=n-1}, we obtain that the injectivity radius is bounded by $\pi$ for $\beta > 1$ and that the injectivity radius is given by half the length of the shortest geodesic loops, $\frac{\ell_\beta}{2}$, for $\beta\in(0,\frac13]\cup[\frac23,1]$.

    For $\beta\in(\frac13,\frac23)$, the bound on the sectional curvature in \cref{thm:sect_curvature} is sharp and cannot be improved, unlike in the case $p = n-1$. In this range, the lower bound on the conjugate radius, $\frac{\pi}{\sqrt{K_\beta}}$, is smaller than half the length of the shortest geodesic loops, $\frac{\ell_\beta}{2}$. Hence, we cannot derive the exact value of the injectivity radius from equations~\eqref{eq:inj=min(l/2,conj)} and~\eqref{eq:conjbound} without knowing the exact value of the conjugate radius. From~\cite[Theorem~5.1]{absilmataigne2024ultimate}, it follows that $\sqrt{2}t_\beta^\mathrm{r}$ is an upper bound on the conjugate radius. Therefore,
    \[
    \frac{\pi}{\sqrt{K_\beta}}\leq \mathrm{conj}_{I_{n\times p}}(\mathrm{St}_\beta(n,p))\leq \sqrt{2}t_\beta^\mathrm{r}.
    \]
    By substituting into equation~\eqref{eq:inj=min(l/2,conj)}, we obtain $\mathrm{inj}(\mathrm{St}_\beta(n,p)) \in [\frac{\pi}{\sqrt{K_\beta}},\min\{\frac{\ell_\beta}{2},\sqrt{2}t_\beta^\mathrm{r}\}]$.
    For a visual illustration of the argument, see \cref{fig:InjRad}.
\end{proof}

\begin{figure}[h!]
  \centering
  \includegraphics[width=0.8\textwidth]{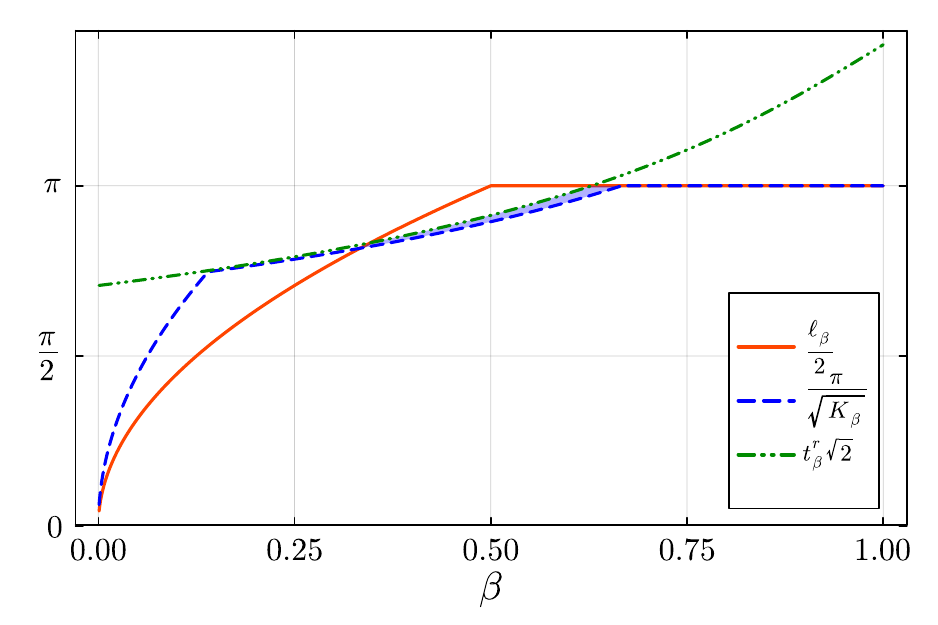}
  \caption{Graphical illustration of half the lengths of shortest geodesic loops $\frac{\ell_\beta}{2} = \min\{\sqrt{2\beta}\pi,\pi\}$ (solid red line), the lower bound on the conjugate radius $\frac{\pi}{\sqrt{K_\beta}}$ (blue dashed line) and the upper bound on the conjugate radius $\sqrt{2}t_\beta^\mathrm{r}$ (green dashed-dotted line) for $\beta\in(0,1]$. The injectivity radius satisfies $\min\{\sqrt{2\beta}\pi,\pi,\frac{\pi}{\sqrt{K_\beta}}\} \leq \mathrm{inj}(\mathrm{St}_\beta(n,p)) \leq \min\{\sqrt{2\beta}\pi,\pi,\sqrt{2}t_\beta^\mathrm{r}\}$. The thin blue shaded area for $\beta\in(\frac13,\frac23)$ is the region where the injectivity radius lies.}
  \label{fig:InjRad}
\end{figure}

For $\beta = 1$, the theorem reproduces the known result for the injectivity radius of the Stiefel manifold endowed with the Euclidean metric, which was previously shown to be $\pi$ in~\cite{zimmermannstoye_eucl_inj:2024}.
For $\beta\in(\frac13,\frac23)$, our result provides only an interval estimate for the injectivity radius, while for $\beta > 1$ we obtain merely an upper bound. This is due to the lack of an explicit expression for the conjugate radius. It is conjectured in~\cite[Conj.~8.1]{absilmataigne2024ultimate} that the upper bound on the injectivity radius is actually sharp, namely that  $\mathrm{inj}(\mathrm{St}_\beta(n,p)) = \min\{\sqrt{2\beta}\pi,\pi,\sqrt{2}t_\beta^\mathrm{r}\}$, for $\beta\in(\frac13,\frac23)$, and  $\mathrm{inj}(\mathrm{St}_\beta(n,p)) = \pi$, for $\beta > 1$. This conjecture is supported by numerical experiments in~\cite{absilmataigne2024ultimate}. The experiments employ an algorithm that investigates various geodesics of a given length. For each such geodesic, the algorithm searches for a shorter geodesic connecting the same endpoints. If a shorter geodesic is found, the original geodesic cannot be length-minimizing, implying that the injectivity radius is smaller than its length. By repeating this process for a range of lengths, the injectivity radius is effectively approached from above.
It should be noted that the evidence is more conclusive for $\beta \leq \frac{1}{2}$, since in that case the algorithm is guaranteed to return whenever the investigated length exceeds the injectivity radius. For $\beta > \frac{1}{2}$, returning is not formally ensured, although it is strongly expected in practice.

Note that the Stiefel manifold equipped with the canonical metric ($\beta = \frac12$) falls in this case where there is only a conjecture on the injectivity radius. We strengthen the conjecture for the canonical metric—namely, that the upper bound for the injectivity radius is sharp—by providing explicit expressions for all Jacobi fields along a specific geodesic of $\mathrm{St}_{\beta=\frac12}(4,2)$ in Appendix \ref{app:first_conjpoint}. These expressions indicate that, when $(n, p) = (4, 2)$ and $\beta = \frac12$, the conjugate points identified in~\cite[Theorem~5.1]{absilmataigne2024ultimate} are indeed first conjugate points.

In total, the injectivity radius of the Stiefel manifold under the $\beta$-metrics is known, or conjectured when $\beta >1$, to coincide with half the length of the shortest geodesic loops for all $\beta$ outside the interval $(\frac13,\frac23)$. Resolving the case $\beta \in (\frac13,\frac23)$ calls for a refined study of the conjugate radius.

\begin{appendices}
\section{Auxiliary lemmata for the proof of \texorpdfstring{\cref{lem:length_bound_phi}}{Lemma 5.7}}\label{app:singval_property}

In this section, we prove the property~\eqref{eq:sing_val_property_phi0pi} of the singular values of $X$ and $\breve{A}$ in the proof of \cref{lem:length_bound_phi} for $\varphi\in(0,\pi)$.
Let $s_1,s_2\in\mathbb{N}$ be positive integers, $\beta > 2$ and let $\widehat{A}\in\mathrm{Skew}(2s_1)$ and $\widehat{B}\in\mathbb{R}^{2s_2\times 2s_1}$ be two nonzero matrices such that they define a geodesic loop $\widehat{\gamma}_\beta(t) \coloneq \mathrm{Exp}_{\beta,U}\left(t\left[\begin{smallmatrix}
    \widehat{A}\\\widehat{B}
\end{smallmatrix}\right]\right)$ with starting point $U\coloneqq I_{2m\times 2s_1}$ and $m\coloneq s_1+s_2$, that satisfies the structure of~\eqref{eq:reduced_loop}, i.e.,
\begin{equation}\label{eq:loop_reminder}
    \exp_{\mathrm{m}}\begin{bmatrix}
        2\beta\widehat{A} & -\widehat{B}^\top\\
        \widehat{B} & 0
    \end{bmatrix} = \begin{bmatrix}
        G_{2 s_1}(\varphi I_{s_1}) & 0\\
        0 & G_{2 s_2}(\varphi I_{s_2})
    \end{bmatrix},\quad\text{and}\quad \exp_{\mathrm{m}}((1-2\beta)\widehat{A}) = G_{2 s_1}(-\varphi I_{s_1}),
\end{equation}
for $\varphi\in(0,\pi)$. 
Define $X\coloneq\begin{bmatrix}
    2\beta\widehat{A} & -\widehat{B}^\top\\
    \widehat{B} & 0
\end{bmatrix}$ and $\breve{A}\coloneq (1-2\beta)\widehat{A}$. 
The singular value pairs of $X$ are~\eqref{eq:singvals_X},
$$|\varphi + 2\pi k_j|,|\varphi + 2\pi k_j|,\text{ for }k_j\in\mathbb{Z}\text{ and }j\in\{1,\dots,m\}.$$
The singular value pairs of $\breve{A}$ are~\eqref{eq:singvals_A},
$$|\varphi + 2\pi\ell_j|,|\varphi+2\pi\ell_j|,\text{ for } \ell_j\in\mathbb{Z} \text{ and } j\in\{1,\dots,s_1\}.$$
We assume, w.l.o.g, that the singular value pairs are ordered by magnitude. 

\begin{lemma}\label[lemma]{lem:singval_property}
    Let $\widehat{\gamma}_\beta(t) \coloneq  \mathrm{Exp}_{\beta,U}\left(t\left[\begin{smallmatrix}
    \widehat{A}\\\widehat{B}
\end{smallmatrix}\right]\right)$ be the geodesic loop from~\eqref{eq:loop_reminder} and let $X\in\mathrm{Skew}(2m)$ and $\breve{A}\in\mathrm{Skew}(2s_1)$ be defined as above.
    Then, there is a $j_a\in\{1,\dots,s_1\}$ such that the difference between the $j_a$-th largest singular value pair of $X$ and the $j_a$-th largest singular value pair of $\breve{A}$ is at least $2\pi$, i.e., equation~\eqref{eq:sing_val_property_phi0pi} holds,
    \begin{align*}
        |\varphi + 2\pi k_{j_a}| - |\varphi + 2\pi\ell_{j_a}|\geq 2\pi.
    \end{align*}
\end{lemma}

The proof of \cref{lem:singval_property} relies on a number of additional results. 
First, we need certain characteristics of $J_{2m}$-orthosymplectic matrices.

\begin{lemma}\label[lemma]{lem:characteristics_J-orthosymplectic}
    Let $m\in\mathbb{N}$ and let $M\in\mathbb{R}^{2m\times 2m}$ be an $J_{2m}$-orthosymplectic matrix. Then, the matrix $M$ has the block structure
    \begin{align}\label{eq:J-orthosymplectic_form}
        M = \begin{bmatrix} 
        M_{1,1} & \cdots & M_{1,m}\\
        \vdots & \ddots & \vdots\\
        M_{m,1} & \cdots & M_{m,m}
        \end{bmatrix},
    \end{align}
    with $2$-by-$2$ blocks $M_{i,j}$, for $1\leq i,j\leq m$, of the form
    \begin{align*}
        M_{i,j} = m_1^{(i,j)}I_2 + m_2^{(i,j)}J_2,
    \end{align*}
    for $m_1^{(i,j)},m_2^{(i,j)}\in\mathbb{R}$.
\end{lemma}
\begin{proof}
    Let $M$ be divided into $2$-by-$2$ blocks according to~\eqref{eq:J-orthosymplectic_form}.
    Since $M$ is $J_{2m}$-orthosymplectic, it commutes with the matrix $J_{2m}$. By the structure of $M$ and $J_{2m}$, every $2$-by-$2$ block $M_{i,j}$ commutes with $J_2$. Hence,
    \begin{align*}
        M_{i,j} = \begin{bmatrix}
            m_1^{(i,j)}&- m_2^{(i,j)}\\
            m_2^{(i,j)} & m_1^{(i,j)}
        \end{bmatrix} =  m_1^{(i,j)}I_2 + m_2^{(i,j)}J_2,
    \end{align*}
    for some $m_1^{(i,j)},m_2^{(i,j)}\in\mathbb{R}$.
\end{proof}

The following corollary is a direct consequence of the structure of a $J_{2m}$-orthosymplectic matrix. 

\begin{corollary}\label[corollary]{cor:orthosymplectic_2x2blockDetails}
    Let $m\in\mathbb{N}$ and let $M\in\mathbb{R}^{2m\times 2m}$ be an $J_{2m}$-orthosymplectic matrix and let $1\leq i\leq m$ be fixed. Then, from the structure~\eqref{eq:J-orthosymplectic_form} of $M$, we obtain
    \begin{align*}
        M_{i,j}M_{i,j}^\top = \alpha_j I_2,
    \end{align*}
    with $\alpha_j \coloneq \left(m_1^{(i,j)}\right)^2 + \left(m_2^{(i,j)}\right)^2\geq 0$, for $1\leq j\leq m$. Furthermore, we obtain from the equation $MM^\top = I_{2m}$ that 
    \begin{align*}
        \sum_{j=1}^{m}M_{i,j}M_{i,j}^\top = I_2.
    \end{align*}
    Therefore, it holds $\sum_{j=1}^m\alpha_j = 1$.
\end{corollary}

Define the diagonal matrices $\widehat{L}$ and $\widehat{K}$ as
\begin{align}
\widehat{L}\coloneq\begin{bmatrix}
        \ell_1 & & \\
         & \ddots & \\
         & & \ell_{s_1}
    \end{bmatrix}\in\mathbb{Z}^{s_1\times s_1},\quad \text{and},\quad \widehat{K}\coloneq\begin{bmatrix}
        k_1  & & \\
         & \ddots & \\
         & & k_{m}
    \end{bmatrix}\in\mathbb{Z}^{m\times m},
\end{align}
and
\begin{align*}
    L\coloneq \widehat{L}\otimes I_2\in\mathbb{Z}^{2s_1\times 2s_1},\quad \text{and},\quad  K\coloneq\widehat{K}\otimes I_2\in\mathbb{Z}^{2m\times 2m}.
\end{align*}
\begin{lemma}\label[lemma]{lem:in_interval_minkjmaxkj}
    Let $\widehat{\gamma}_\beta(t) =  \mathrm{Exp}_{\beta,U}\left(t\left[\begin{smallmatrix}
    \widehat{A}\\\widehat{B}
    \end{smallmatrix}\right]\right)$ be the geodesic loop from~\eqref{eq:loop_reminder} with $\beta > 2$ and $\varphi\in(0,\pi)$. Let, furthermore, $X\in\mathrm{Skew}(2m)$, $\breve{A}\in\mathrm{Skew}(2s_1)$ and their singular values be defined as above. 
    Then, it holds
    \begin{align}\label{eq:in_interval[minkj,maxkj]}
        \frac{1}{2\beta-1}\frac{\varphi}{2\pi} + \frac{2\beta}{2\beta-1}\ell_1 \in [\min_j k_j,\max_j k_j].
    \end{align}
\end{lemma}

\begin{proof}
    We start by showing that there is a $J_{2m}$-orthosymplectic matrix $M_X$ and a $J_{2s_1}$-orthosymplectic matrix $M_A$ such that
    \begin{align}\label{eq:MKM^T}
        \begin{bmatrix}
            \frac{1}{2\beta-1}\frac{\varphi}{2\pi}I_{2s_1}+\frac{2\beta}{2\beta-1}L & -\frac{1}{2\pi}J_{2s_1}^\top M_A^\top\widehat{B}^\top\\
            \frac{1}{2\pi}J_{2s_2}^\top\widehat{B}M_A & -\frac{\varphi}{2\pi}I_{2s_2}
        \end{bmatrix} = M_XKM_X^\top,
    \end{align}
    where $L$ and $K$ are defined as above. 
    Next, we obtain the desired equation~\eqref{eq:in_interval[minkj,maxkj]} from the upper left $2$-by-$2$ block of equation~\eqref{eq:MKM^T}.
    
    By the definition of the geodesic loop $\widehat{\gamma}_\beta$, it holds
    \begin{align*}
        \breve{A} = (1-2\beta)\widehat{A}\in\exp_\mathrm{m}^{-1}(G_{2s_1}(-\varphi I_{s_1})).
    \end{align*}
    Hence, by \cref{lem:logGphiI} and the definition of $\widehat{L}$ there is a $J_{2s_1}$-orthosymplectic matrix $M_A$ such that
    \begin{align*}
        \breve{A} = M_A\Omega_{2s_1}(-\varphi I_{s_1},-\widehat{L})M_A^\top.
    \end{align*}
    So,
    \begin{align*}
        M_A^\top \breve{A}M_A = \Omega_{2s_1}(-\varphi I_{s_1},-\widehat{L}) = J_{2s_1}(-\varphi I_{2s_1} - 2\pi L).
    \end{align*}
    In the middle term, the parentheses hold the input argument to the function $\Omega_{2s_1}$, while the right term is a matrix product. When applying the orthogonal transformation by $\begin{bmatrix}
        M_A & 0 \\
        0 & I_{2s_2}
    \end{bmatrix}$ to the matrix $X$, we obtain 
    \begin{align}
        \Breve{X} \coloneq& \begin{bmatrix}
            M_A^\top & 0\\
            0 & I_{2s_2}
        \end{bmatrix}X\begin{bmatrix}
            M_A & 0 \\
            0 & I_{2s_2}
        \end{bmatrix} = \begin{bmatrix}
            M_A^\top & 0\\
            0 & I_{2s_2}
        \end{bmatrix}\begin{bmatrix}
            2\beta\widehat{A} & -\widehat{B}^\top\\
            \widehat{B} & 0
        \end{bmatrix}\begin{bmatrix}
            M_A & 0 \\
            0 & I_{2s_2}
        \end{bmatrix}\notag\\
        =& \begin{bmatrix}
            \frac{2\beta}{1-2\beta}J_{2s_1}(-\varphi I_{2s_1} - 2\pi L) & -M_A^\top\widehat{B}^\top\\
            \widehat{B}M_A & 0
        \end{bmatrix}\notag\\
        =& J_{2m}\begin{bmatrix}
            \frac{2\beta}{2\beta-1}(\varphi I_{2s_1} + 2\pi L) & -J_{2s_1}^\top M_A^\top \widehat{B}^\top\\
            J_{2s_2}^\top \widehat{B} M_A & 0
        \end{bmatrix}\label{eq:X_T_F1}.
    \end{align}
    The matrix $\Breve{X}$ has the same singular values as $X$. 
    From the proof of \cref{lem:logGphiI} we know that $M_A$ is in the invariance group of $G_{2s_1}(-\varphi I_{2s_1})$, i.e., $M_A\in\mathrm{ig}(G_{2s_1}(-\varphi I_{2s_1}))$. It is easy to see that $M_A$ is also in the invariance group of $G_{2s_1}(\varphi I_{2s_1})$. Therefore, it holds that like $X$ also $\Breve{X}$ is a matrix exponential inverse of $G_{2m}(\varphi I_{2m})$,
    \begin{align*}
        \exp_{\mathrm{m}}(\Breve{X}) &= \begin{bmatrix}
            M_A^\top & 0\\
            0 & I_{2s_2}
        \end{bmatrix}\exp_{\mathrm{m}}(X)\begin{bmatrix}
            M_A & 0\\
            0 & I_{2s_2}
        \end{bmatrix}\\
        &= \begin{bmatrix}
            M_A^\top G_{2s_1}(\varphi I_{s_1})M_A & 0\\
            0 & G_{2s_2}(\varphi I_{s_2})
        \end{bmatrix} = \begin{bmatrix}
            G_{2s_1}(\varphi I_{s_1}) & 0\\
            0 & G_{2s_2}(\varphi I_{s_2})
        \end{bmatrix}\\
        &= \exp_{\mathrm{m}}(X).
    \end{align*}
    Therefore, by \cref{lem:logGphiI} there is a $J_{2m}$-orthosymplectic matrix $M_X$ such that
    \begin{align*}
        \Breve{X} = M_X\Omega_{2m}(\varphi I_{m},\widehat{K})M_X^\top.
    \end{align*}
    So, 
    \begin{align}
        \Breve{X} &= M_X\Omega_{2m}(\varphi I_{m},\widehat{K})M_X^\top = M_X J_{2m} (\varphi I_{2m} + 2\pi K) M_X^\top\notag\\
        &= J_{2m}(\varphi I_{2m} + 2\pi M_X K M_X^\top).\label{eq:X_T_F2}
    \end{align}
    Combining the two formulas~\eqref{eq:X_T_F1} and~\eqref{eq:X_T_F2} for $\Breve{X}$, we obtain
    \begin{align*}
        \begin{bmatrix}
            \frac{2\beta}{2\beta-1}(\varphi I_{2s_1} + 2\pi L) & -J_{2s_1}^\top M_A^\top \widehat{B}^\top\\
            J_{2s_2}^\top \widehat{B} M_A & 0
        \end{bmatrix} = \varphi I_{2m} + 2\pi M_X K M_X^\top.
    \end{align*}
    Therefore, we obtain equation~\eqref{eq:MKM^T}
    \begin{align*}
        \begin{bmatrix}
            \frac{1}{2\beta-1}\frac{\varphi}{2\pi}I_{2s_1}+\frac{2\beta}{2\beta-1}L & -\frac{1}{2\pi}J_{2s_1}^\top M_A^\top\widehat{B}^\top\\
            \frac{1}{2\pi}J_{2s_2}^\top\widehat{B}M_A & -\frac{\varphi}{2\pi}I_{2s_2}
        \end{bmatrix} = M_XKM_X^\top.
    \end{align*}
    Next, we extract the upper left $2$-by-$2$ block. To this end, we write the $J_{2m}$-orthosymplectic matrix $M_X$ in the block structure~\eqref{eq:J-orthosymplectic_form} from \cref{lem:characteristics_J-orthosymplectic},
    \begin{align*}
        M_X = \begin{bmatrix}
            M_{1,1} & \cdots & M_{1,m}\\
            \vdots & \ddots & \vdots\\
            M_{m,1} & \cdots & M_{m,m}
        \end{bmatrix}.
    \end{align*}
    Now, the upper left $2$-by-$2$ block of equation~\eqref{eq:MKM^T} is given by
    \begin{align*}
        \left(\frac{1}{2\beta-1}\frac{\varphi}{2\pi} + \frac{2\beta}{2\beta-1}\ell_1\right)I_2 &= I_{2m,2}^\top M_X K M_X^\top I_{2m,2}\\
        &= \sum_{j=1}^{m}M_{(1,j)}(k_jI_2)M_{(1,j)}^\top\\
        &= \sum_{j=1}^{m}k_jM_{(1,j)}M_{(1,j)}^\top.
    \end{align*}
    From \cref{cor:orthosymplectic_2x2blockDetails} (with $i=1$), we know that the products $M_{(1,j)}M_{(1,j)}^\top$ are diagonal and that there are coefficients $\alpha_j\geq 0$ such that
    \begin{align*}
        \left(\frac{1}{2\beta-1}\frac{\varphi}{2\pi} + \frac{2\beta}{2\beta-1}\ell_1\right)I_2 &= \sum_{j=1}^{m}k_j(\alpha_j I_2)
    \end{align*}
    and $\sum_{j=1}^{m} \alpha_j = 1$. As a direct consequence, the term $\frac{1}{2\beta-1}\frac{\varphi}{2\pi} + \frac{2\beta}{2\beta-1}\ell_1$ can be obtained by a convex combination of the integers $k_1,\dots,k_{m}$. Therefore,
    \begin{align*}
        \frac{1}{2\beta-1}\frac{\varphi}{2\pi} + \frac{2\beta}{2\beta-1}\ell_1\in[\min_j k_j,\max_j k_j].
    \end{align*}
    This concludes the proof. 
\end{proof}

To prove the existence of a $j_a\in\{1,\dots,s_1\}$ such that equation~\eqref{eq:sing_val_property_phi0pi} holds,
\begin{align*}
    |\varphi + 2\pi k_{j_a}| - |\varphi + 2\pi\ell_{j_a}|\geq 2\pi,
\end{align*}
we assume the opposite and lead this to a contradiction. \cref{lem:cond_on_k_l_for_contradiction,lem:convexcombi_imposible_under_assumption} will be used to prove that under the assumption that
\begin{align*}
    |\varphi + 2\pi k_{j}| - |\varphi + 2\pi\ell_{j}| < 2\pi
\end{align*} 
holds for all $j\in\{1,\dots,s_1\}$, the equation~\eqref{eq:in_interval[minkj,maxkj]} is not fulfilled. This gives us the desired contradiction.

\begin{lemma}\label[lemma]{lem:cond_on_k_l_for_contradiction}
    Let $\varphi\in(0,\pi)$ and let $k_j,\ell_j\in\mathbb{Z}$ be two integers. Assume that it holds 
    \begin{align}\label{eq:helper_singvaldiff_small}
        |\varphi + 2\pi\ell_j| < |\varphi + 2\pi k_j|.
    \end{align}
    Then, the inequality 
    \begin{align}\label{eq:singvaldiff_small}
        |\varphi + 2\pi k_j| - |\varphi + 2\pi\ell_j| < 2\pi
    \end{align}
    holds only true if $\ell_j \geq 0$ is non-negative and $k_j = -\ell_j-1$ or if $\ell_j < 0$ and $k_j = -\ell_j$.
\end{lemma}
\begin{proof}
    We start with the case where $\ell_j\geq 0$ is non-negative. In this case the term $\varphi + 2\pi\ell_j > 0$ is positive and equations~\eqref{eq:helper_singvaldiff_small} and~\eqref{eq:singvaldiff_small} give 
    \begin{align}\label{eq:l>0}
       \varphi + 2\pi\ell_j <|\varphi + 2\pi k_j| < \varphi + 2\pi(\ell_j+1).
    \end{align}
    For $\varphi + 2\pi k_j\geq 0$, equation~\eqref{eq:l>0} simplifies to $\ell_j< k_j<\ell_j+1$, which is impossible for $k_j\in\mathbb{Z}$.  For $\varphi + 2\pi k_j< 0$, and thus $k_j<0$, equation~\eqref{eq:l>0} gives 
    \begin{equation*}
     \ell_j <\frac{-\varphi}{\pi} + |k_j| < \ell_j+1\Longrightarrow k_j = -  \ell_j-1.
    \end{equation*}
    If $\varphi + 2\pi\ell_j < 0$ and thus $\ell_j < 0$, equations~\eqref{eq:helper_singvaldiff_small} and~\eqref{eq:singvaldiff_small} give 
    \begin{align}
       -\varphi + 2\pi|\ell_j| <|\varphi + 2\pi k_j| < -\varphi + 2\pi(|\ell_j|+1).
    \end{align}
    Similarly, $\varphi + 2\pi k_j<0$ is impossible and $\varphi + 2\pi k_j\geq 0$ yields $k_j = -\ell_j$.
\end{proof}

\begin{lemma}\label[lemma]{lem:convexcombi_imposible_under_assumption}
    Let $\varphi\in(0,\pi)$, $\beta > 2$, $m\in\mathbb{N}_{\geq 2}$ and let $\ell_1,k_1,\dots,k_m\in\mathbb{Z}$ be integers such that $|\varphi + 2\pi k_1|\geq \dots \geq |\varphi + 2\pi k_m|$. Furthermore, assume that if $\ell_1 \geq 0$ is non-negative, then $k_1 = -\ell_1-1$, and if $\ell_1 < 0$ is negative, then $k_1 = -\ell_1$. 
    Then,
    \begin{align}\label{eq:term_notin_[mink,maxk]}
        \frac{1}{2\beta-1}\frac{\varphi}{2\pi} + \frac{2\beta}{2\beta-1}\ell_1\notin [\min_j k_j, \max_j k_j].
    \end{align}
\end{lemma} 
\begin{proof}
    With $|\varphi + 2\pi k_1|\geq\dots\geq|\varphi + 2\pi k_{m}|$ and $\varphi\in(0,\pi)$, it also holds $|k_1|\geq\dots\geq|k_{m}|$.

    First, we tackle the case where $\ell_1 < 0$ and $k_1 = -\ell_1~(> 0)$. By $|k_1|\geq \dots\geq |k_m|$, it holds $\min_j k_j\geq-|k_1| = \ell_1$. Hence, it is sufficient to show that $\frac{1}{2\beta-1}\frac{\varphi}{2\pi} + \frac{2\beta}{2\beta-1}\ell_1<\ell_1$. 
    With $\frac{\varphi}{2\pi}<\frac12<|\ell_1|$, we have $\frac{\varphi}{2\pi} + \ell_1 < 0$. So,
    \begin{align*}
        \frac{1}{2\beta-1}\frac{\varphi}{2\pi} + \frac{2\beta}{2\beta-1}\ell_1 = \frac{1}{2\beta-1}\left(\frac{\varphi}{2\pi} + \ell_1\right) + \ell_1 < \ell_1.
    \end{align*}
    Next, we tackle the case where $\ell_1\geq 0$ and $k_1 = -\ell_1 - 1~(< 0)$. By $|k_1|\geq \dots\geq |k_m|$, it holds $k_{\mathrm{max}} \coloneq \max_j k_j \leq |k_1| = \ell_1 + 1$. If we assume $k_{\mathrm{max}} = |k_1|$, we have $k_{\mathrm{max}} \neq k_1$ and therefore, according to the requirements of this lemma, $|\varphi + 2\pi k_1| \geq |\varphi + 2\pi k_{\mathrm{max}}|$. This leads to the contradiction $-\varphi\geq \varphi$. Hence, $k_{\mathrm{max}}\leq |k_1| - 1 = \ell_1$ and it is sufficient to show that $\frac{1}{2\beta-1}\frac{\varphi}{2\pi} + \frac{2\beta}{2\beta-1}\ell_1 > \ell_1$. With $\varphi > 0$, we have
    \begin{align*}
        \frac{1}{2\beta-1}\frac{\varphi}{2\pi} + \frac{2\beta}{2\beta-1}\ell_1 > \frac{2\beta}{2\beta-1}\ell_1 \geq \ell_1.
    \end{align*}
\end{proof}

Now, we formulate the proof of \cref{lem:singval_property} with the help of the previous lemmata.
\begin{proof}(\cref{lem:singval_property})

    Let us recall that we want to prove that, given the geodesic loop $\widehat{\gamma}_\beta$ and the corresponding matrices $X$ and $\breve{A}$, there exists a $j_a\in\{1,\dots,s_1\} $ such that the difference between the $j_a$-largest singular value pair of $X$ and the $j_a$-largest singular value pair of $\breve{A}$ is at least $2\pi$., i.e., equation~\eqref{eq:sing_val_property_phi0pi} holds,
    \begin{align*}
        |\varphi + 2\pi k_{j_a}| - |\varphi + 2\pi\ell_{j_a}|\geq 2\pi.
    \end{align*}
    First, by \cref{lem:in_interval_minkjmaxkj}, the existence of the geodesic loop $\widehat{\gamma}_\beta$ leads to the equation~\eqref{eq:in_interval[minkj,maxkj]}
    \begin{align*}
        \frac{1}{2\beta-1}\frac{\varphi}{2\pi} + \frac{2\beta}{2\beta-1}\ell_1\in[\min_j k_j, \max_j k_j].
    \end{align*}
    To prove the existence of a $j_a\in\{1,\dots,s_1\}$ such that equation~\eqref{eq:sing_val_property_phi0pi} holds, we assume the contrary and lead this to a contradiction. 
    So, we assume that 
    \begin{align}\label{eq:assumption_for_contradiction}
        |\varphi + 2\pi k_{j}| - |\varphi + 2\pi\ell_{j}| < 2\pi
    \end{align} 
    holds for all $j\in\{1,\dots,s_1\}$. By the interlacing property of singular values (\cref{lem:interlacing}), we obtain that 
    \begin{align*}
        |\varphi + 2\pi\ell_{j}| < \frac{2\beta}{2\beta-1}|\varphi + 2\pi\ell_{j}|\leq |\varphi + 2\pi k_{j}|
    \end{align*}
    holds for all $j\in\{1,\dots,s_1\}$. Therefore, we are able to apply  \cref{lem:cond_on_k_l_for_contradiction} (for $j=1$) to obtain that $k_1 = -\ell_1-1$, if $\ell_1\geq 0$, and $k_1 = -\ell_1$, if $\ell_1 < 0$. Hence, the requirements of \cref{lem:convexcombi_imposible_under_assumption} are fulfilled and we obtain that the contrary to equation~\eqref{eq:in_interval[minkj,maxkj]} holds. This gives the desired contradiction and concludes the proof.
\end{proof}

\section{Bounds on the sectional curvature of the Stiefel manifold under the family of \texorpdfstring{$\beta$}{beta}-metrics for \texorpdfstring{$p=n-1$}{p=n-1} and \texorpdfstring{$\beta\in(\frac13,\frac23)$}{beta in (1/3, 2/3)}}\label{sec:injradp=n-1}

From \cref{thm:sect_curvature}, we obtain bounds on the sectional curvature of the Stiefel manifold $\mathrm{St}_\beta(n,p)$. When $p = n-1$, these bounds are not sharp and can be improved. In view of the proof of \cref{thm:injrad_p=n-1}, we need sharper bounds for $\beta\in(\frac13,\frac23)$ in order to determine the injectivity radius for all values of $\beta\in(0,1]$.

\begin{theorem}
    Let $\beta\in(\frac13,\frac23)$. The sectional curvature of the Stiefel manifold $\mathrm{St}_\beta(n,n-1)$ is bounded above by $K_\beta \coloneq \frac{1}{2\beta}$.
\end{theorem}
\begin{proof}
    For a $\beta$-orthonormal basis $\{\Delta_1,\Delta_2\} \in T_{I_{n\times (n-1)}}\mathrm{St}_\beta(n,n-1)$ of a two-dimensional tangent space section, the matrix blocks are related by 
    $\|\Delta_i\|_\beta^2 = \beta \|A_i\|_\mathrm{F}^2 + \|B_i\|_\mathrm{F}^2 =1$ for $i=1,2$ and
    $\langle \Delta_1,\Delta_2\rangle_\beta =\beta \mathrm{tr}(A_1^\top A_2) + \mathrm{tr}(B_1^\top B_2)=0$. Furthermore, $B_1,B_2$ are $1$-by-$(n-1)$ dimensional matrices and therefore $\|B_1B_2^\top - B_2B_1^\top\|_\mathrm{F} = 0$. A reformulation of the formula for the sectional curvature~\eqref{eq:sect_curv_formula} leads to
    \begin{align*}
        \mathcal{K}_\beta(\Delta_1,\Delta_2) =& \frac{2-3\beta}{4}\|B_1^\top B_2-B_2^\top B_1\|_\mathrm{F}^2 + \beta^2\|B_1A_2 - B_2A_1\|_\mathrm{F}^2\\
        &+\frac{\beta}{4}\|[A_1,A_2]\|_\mathrm{F}^2 - \frac{\beta}{2}(3-4\beta)\mathrm{tr}\left([A_1,A_2]^\top(B_1^\top B_2 - B_2^\top B_1)\right),
    \end{align*}
    see also~\cite[eq.~(35)]{nguyen2022curvature}. As explained in~\cite{zimmermannstoye_curvature:2024}, we have the following term-by-term estimates,
    where $\alpha_i = \|A_i\|_F$, $\eta_i = \|B_i\|_F$:
    \begin{itemize}
        \item $\frac12\|B_1^\top B_2 - B_2^\top B_1\|_\mathrm{F}^2\leq \eta_1^2\eta_2^2$,
        \item $\beta^2\|B_1A_2 - B_2A_1\|_\mathrm{F}^2\leq \frac{\beta^2}2\eta_1^2\alpha_2^2 + \frac{\beta^2}2 \eta_2^2 \alpha_1^2 + \beta^2\alpha_1\alpha_2\eta_1\eta_2$,
        \item $\| [A_1,A_2]\|_\mathrm{F}^2\leq \alpha_1^2\alpha_2^2$,
        \item $|\frac{\beta}{2}(3-4\beta)\mathrm{tr}\left([A_1,A_2]^\top(B_1^\top B_2-B_2^\top B_1)\right)| \leq \frac{\beta(3-4\beta)}{\sqrt{2}}\alpha_1\alpha_2\eta_1\eta_2$.
    \end{itemize}
    We start by using the term-by-term estimates and the relation $\alpha_i^2 = \frac{1}{\beta}(1-\eta_i^2)$ to obtain the upper bound
    \begin{align}
        \mathcal{K}_\beta(\Delta_1,\Delta_2) \leq& \frac{2-3\beta}{2}\eta_1^2\eta_2^2 + \frac{\beta^2}2(\eta_1^2\alpha_2^2 + \eta_2^2 \alpha_1^2) \notag\\
        &+ \frac{\beta}{4}\alpha_1^2\alpha_2^2 + \left(\beta^2 + \frac{\beta(3-4\beta)}{\sqrt{2}}\right)\alpha_1\alpha_2\eta_1\eta_2 \notag \\
        \label{eq:intermediate_upperbound}
        \begin{split}
        =& \left(\frac{2-3\beta}{2} - \beta + \frac{1}{4\beta}\right)\eta_1^2\eta_2^2 + \left(\frac{2\beta^2-1}{4\beta}\right)(\eta_1^2 + \eta_2^2)\\
        &+\left(\beta^2 + \frac{\beta(3-4\beta)}{\sqrt{2}}\right)\alpha_1\alpha_2\eta_1\eta_2 + \frac{1}{4\beta}.
        \end{split}
    \end{align}
    The factor in front of the $\alpha_1\alpha_2\eta_1\eta_2$-term is positive and the factor in front of the $(\eta_1^2 + \eta_2^2)$-term is negative (for $\beta\in(\frac13,\frac23)$). Therefore, we may use the estimate $t_1t_2\leq\frac12(t_1^2+t_2^2)$ and $\alpha_i^2 = \frac{1}{\beta}(1-\eta_i^2)$ to obtain:
    \begin{itemize}
        \item $\alpha_1\alpha_2\eta_1\eta_2 \leq \frac12(\alpha_1^2 + \alpha_2^2)\eta_1\eta_2 = \frac{1}{2\beta}(2-(\eta_1^2+\eta_2^2))\eta_1\eta_2\leq\frac{1}{\beta}(1-\eta_1\eta_2)\eta_1\eta_2$,

        \item $\left(\frac{2\beta^2-1}{4\beta}\right)(\eta_1^2 + \eta_2^2) \leq \left(\frac{2\beta^2-1}{2\beta}\right)\eta_1\eta_2$.
    \end{itemize}
    Hence, we obtain an upper bound on the sectional curvature given by
    \begin{align*}
       ~\eqref{eq:intermediate_upperbound} \leq& \left(\frac{2-3\beta}{2} - \beta + \frac{1}{4\beta}\right)\eta_1^2\eta_2^2 + \left(\frac{2\beta^2-1}{2\beta}\right)\eta_1\eta_2\\
        &+\left(\beta + \frac{(3-4\beta)}{\sqrt{2}}\right)(\eta_1\eta_2 - \eta_1^2\eta_2^2) + \frac{1}{4\beta}\\
        =& -\underbrace{\frac{1}{4\beta}\left(2\beta^2(7-4\sqrt{2})+2\beta(3\sqrt{2}-2)-1\right)}_{=:a_\beta}x^2 \\
        &+ \underbrace{\frac{1}{2\beta}\left(2\beta^2(2-2\sqrt{2}) + 3\sqrt{2}\beta - 1\right)}_{=:b_\beta}x + \frac{1}{4\beta},
    \end{align*}
    for $x\coloneq \eta_1\eta_2\in[0,1]$. It holds $a_\beta > 0$ for $\beta\in(\frac13,\frac32)$. So, this is a parabola $p_\beta(x)$ in $x$ with a downward opening. The parabola attains its maximum value at $x_{\max} = \frac12\frac{b_\beta}{a_\beta}$.
    In order to show that the sectional curvatures of $\mathrm{St}_\beta(n,p)$ are bounded by $\frac{1}{2\beta}$, for $\beta\in(\frac13,\frac23)$, it is sufficient to show that 
    \begin{align*}
        p_\beta(x_{\max}) = \frac{1}{4\beta} + \frac14\frac{b_\beta^2}{a_\beta}\leq \frac{1}{2\beta}.
    \end{align*}
    With $a_\beta > 0$, we obtain equivalently $b_\beta^2 - \frac{a_\beta}{\beta}\leq 0$ and it can be shown that this is equivalent a problem of showing that a polynomial of degree $4$ is non-positive for $\beta\in(\frac13,\frac23)$,
    \begin{align*}
        \beta^4(48-32\sqrt{2}) + \beta^3(24\sqrt{2}-48) + \beta^2(16\sqrt{2}-4) + \beta(4-12\sqrt{2}) + 2\leq 0.
    \end{align*}
    One can prove that the polynomial is monotonically decreasing in the $\beta$-interval of interest. Therefore, the quartic polynomial is smaller than its value at the left boundary ($\beta = \frac13$), which is $\frac{2}{81}(69-70\sqrt{2})\approx -0.74$. Hence, the quartic polynomial is negative in the $\beta$-interval of interest and the inequality holds true. In total, we obtain that the sectional curvature is bounded above by $K_\beta\coloneq \frac{1}{2\beta}$, which concludes the proof.
\end{proof}

Note that it remains an open problem whether this bound on the sectional curvature is sharp.

\section{A first conjugate point on \texorpdfstring{$\mathrm{St}_\frac12(4,2)$}{St(4,2) for beta=1/2}}\label{app:first_conjpoint}
The conjugate radius is given by the minimum of the geodesic distances between the starting point and the first conjugate point of geodesics on the manifold. The Theorem 5.1 in~\cite{absilmataigne2024ultimate} gives conjugate points at a geodesic distance of $\sqrt{2}t_\beta^\mathrm{r}$ with $t_\beta^\mathrm{r}$ being the first positive root of $\frac{\sin(t)}{t} + \frac{1-\beta}{\beta}\cos(t)$. This provides an upper bound on the conjugate radius. In this section, we show that the conjugate point found in~\cite[Theorem 5.1]{absilmataigne2024ultimate} describes a first conjugate point on $\mathrm{St}_\frac12(4,2)$. This strengthens the conjecture made in~\cite[Conj. 8.1]{absilmataigne2024ultimate} that the upper bound on the conjugate radius is sharp. 
Since the Stiefel manifold is a homogeneous space, we choose the starting point of the geodesic to be at $U\coloneq I_{4\times 2}$. The initial velocity of the geodesic considered in~\cite[Theorem 5.1]{absilmataigne2024ultimate} is given by $\begin{bmatrix}
    A\\B
\end{bmatrix}\in T_{I_{4\times 2}}\mathrm{St}_\frac12(4,2)$, with $A = 0$ and $B = \frac{1}{\sqrt{2}} I_2$ when parameterized by its arc-length.
So, we consider the geodesic $\gamma(t) \coloneq \mathrm{Exp}_{\frac12,I_{4\times 2}}\left(t\begin{bmatrix}
    A\\B
\end{bmatrix}\right)$. 

Conjugate points are characterized as zeros of special vector fields along a geodesic, the so-called Jacobi fields. Jacobi fields, in turn, are determined by an ordinary second-order differential equation. For a geodesic $\gamma$ and a tangent vector $W\in T_{\gamma^\prime(0)}(T_U \mathrm{St}_\frac12(4,2))$, the associated Jacobi field takes the explicit form of 
\begin{align}\label{eq:Jacobifield}
    J(t) \coloneq (\mathrm{d}\mathrm{Exp}_{\frac12,U})_{t\gamma^\prime(0)}(tW),
\end{align}
cf.~\cite[Chap. 5, Cor. 2.5]{docarmo}. By definition~\cite[Chap. 5, Def. 3.1]{docarmo}, a point $\gamma(t_1)$ is conjugate to the starting point $\gamma(0) = U$ along the geodesic if there exists a non-zero Jacobi field along $\gamma$ such that $J(0) = J(t_1) = 0$. 

To determine conjugate points along the geodesic $\gamma$, we need to calculate the directional derivative of the Riemannian exponential on $T_{I_{4\times 2}}\mathrm{St}_\frac12(4,2)$ with base at $t\gamma^\prime(0)$ in the direction $tW$. We follow the approach of~\cite[Section 4.2]{ZimmermannHermite_2020}.
The tangent $\gamma^\prime(0) = \begin{bmatrix}
    A\\B
\end{bmatrix}$ is given by the matrices $A$ and $B$. Since a tangent space to a vector space can be identified with the vector space itself (see e.g.~\cite[p.13]{grassmann}), $T_{\gamma^\prime(0)}(T_{I_{4\times 2}}\mathrm{St}_\frac12(4,2))\cong T_{I_{4\times 2}}\mathrm{St}_\frac12(4,2)$, the direction of variation $W\in T_{\gamma^\prime(0)}(T_{I_{4\times 2}}\mathrm{St}_\frac12(4,2))$ can be parameterized analogously by some $A_w\in\mathrm{Skew}(2)$ and $B_w\in\mathbb{R}^{2\times 2}$.
According to the formula for the Riemannian exponential of the Stiefel manifold with respect to the canonical metric
\begin{align*}
    \mathrm{Exp}_{\frac12,I_{4,2}}(t\gamma^\prime(0)) = \exp_{\mathrm{m}}\left(t\begin{bmatrix}
        A & -B^\top\\
        B & 0
    \end{bmatrix}\right)I_{4\times 2},
\end{align*}
the calculation of its directional derivative boils down to extracting the first two columns from the directional derivative of the matrix exponential at 
\begin{scriptsize}
$t
    \begin{pmatrix} A & -B^\top\\B&0
    \end{pmatrix}
$
\end{scriptsize} in the direction 
\begin{scriptsize}
$t\begin{pmatrix} A_w & -B_w^\top\\B_w&0\end{pmatrix}$
\end{scriptsize}. Najfeld and Havel~\cite{NAJFELD1995321} provide a formula for calculating the directional derivative of the matrix exponential. The directional derivative $D\exp_\mathrm{m}(X)[Y]$ can be calculated via
\begin{align}\label{eq:expmDirectDeriv}
\exp_\mathrm{m}\left(\begin{pmatrix}X&Y\\0&X\end{pmatrix}\right) = \begin{pmatrix}\exp_\mathrm{m}(X)&\mathrm{D}\exp_\mathrm{m}(X)[Y]\\0&\exp_\mathrm{m}(X)\end{pmatrix}.
\end{align}
This goes by the name of Mathias' Theorem in~\cite[Theorem 3.6]{Higham:2008:FM}.
Next, we explicitly write down a basis of the vector space of Jacobi fields along the geodesic $\gamma$. Those basis elements involve trigonometric functions. The problem of finding the first conjugate point along the geodesic is then reduced to the problem of finding the smallest positive zero of those trigonometric functions. 
By a standard result on Jacobi fields, see e.g.~\cite{docarmo}, there are $5 = \text{dim}(\mathrm{St}_\frac12(4,2))$ linearly independent Jacobi fields along $\gamma$ when the starting point $J(0) = 0$ is fixed. They can be obtained from~\eqref{eq:Jacobifield} by choosing linearly independent directions $W$, see~\cite[Chap. 5, Remark 3.2]{docarmo}. Five linearly independent directions $W_i \in T_{U_0}St(n,p)$ are 
\begin{align*}
	W_1 &= U\begin{pmatrix}0&-1\\1&0\end{pmatrix},\quad W_2 = U^\perp\begin{pmatrix}1&-1\\1&1\end{pmatrix}, \quad  W_3 = U^\perp\begin{pmatrix}1&1\\-1&1\end{pmatrix},\\
	W_4 &= U^\perp\begin{pmatrix}1&1\\1&-1\end{pmatrix},\quad 
	W_5 = U^\perp\begin{pmatrix}-1&1\\1&1\end{pmatrix}. 
\end{align*}
Calculating the Jacobi fields via the formula for the directional derivative of the matrix exponential~\eqref{eq:expmDirectDeriv} can be done with the Jordan canonical form. We obtain 
\begin{align*}
	J_1(t) &=(\mathrm{d}\mathrm{Exp}_{\frac12,I_{4\times 2}})_{t\gamma^\prime(0)}[tW_1]= \begin{pmatrix}
		0&-\frac{t\cos\left(\frac{t}{\sqrt{2}}\right) + \sqrt{2}\sin\left(\frac{t}{\sqrt{2}}\right)}{2}\\
		\frac{t\cos\left(\frac{t}{\sqrt{2}}\right) + \sqrt{2}\sin\left(\frac{t}{\sqrt{2}}\right)}{2} & 0\\
		0&-\frac{t}{2}\sin\left(\frac{t}{\sqrt{2}}\right)\\
		\frac{t}{2}\sin\left(\frac{t}{\sqrt{2}}\right)&0
	\end{pmatrix},\\
	J_2(t) &=(\mathrm{d}\mathrm{Exp}_{\frac12,I_{4\times 2}})_{t\gamma^\prime(0)}[tW_2]= \begin{pmatrix}
		-t\sin\left(\frac{t}{\sqrt{2}}\right)&0\\
		0&-t\sin\left(\frac{t}{\sqrt{2}}\right)\\
		t\cos\left(\frac{t}{\sqrt{2}}\right)&-\sqrt{2}\sin\left(\frac{t}{\sqrt{2}}\right)\\
		\sqrt{2}\sin\left(\frac{t}{\sqrt{2}}\right)&t\cos\left(\frac{t}{\sqrt{2}}\right)
	\end{pmatrix},\\
	J_3(t) &=(\mathrm{d}\mathrm{Exp}_{\frac12,I_{4\times 2}})_{t\gamma^\prime(0)}[tW_3]= \begin{pmatrix}
		-t\sin\left(\frac{t}{\sqrt{2}}\right)&0\\
		0&-t\sin\left(\frac{t}{\sqrt{2}}\right)\\
		t\cos\left(\frac{t}{\sqrt{2}}\right)&\sqrt{2}\sin\left(\frac{t}{\sqrt{2}}\right)\\
		-\sqrt{2}\sin\left(\frac{t}{\sqrt{2}}\right)&t\cos\left(\frac{t}{\sqrt{2}}\right)
	\end{pmatrix},\\
	J_4(t) &=(\mathrm{d}\mathrm{Exp}_{\frac12,I_{4\times 2}})_{t\gamma^\prime(0)}[tW_4]= \begin{pmatrix}
		-t\sin\left(\frac{t}{\sqrt{2}}\right)&-t\sin\left(\frac{t}{\sqrt{2}}\right)\\
		-t\sin\left(\frac{t}{\sqrt{2}}\right)&t\sin\left(\frac{t}{\sqrt{2}}\right)\\
		t\cos\left(\frac{t}{\sqrt{2}}\right)&t\cos\left(\frac{t}{\sqrt{2}}\right)\\
		t\cos\left(\frac{t}{\sqrt{2}}\right)&-t\cos\left(\frac{t}{\sqrt{2}}\right)
	\end{pmatrix},\\ 
	J_5(t) &=(\mathrm{d}\mathrm{Exp}_{\frac12,I_{4\times 2}})_{t\gamma^\prime(0)}[tW_5]= \begin{pmatrix}
		t\sin\left(\frac{t}{\sqrt{2}}\right)&-t\sin\left(\frac{t}{\sqrt{2}}\right)\\
		-t\sin\left(\frac{t}{\sqrt{2}}\right)&-t\sin\left(\frac{t}{\sqrt{2}}\right)\\
		-t\cos\left(\frac{t}{\sqrt{2}}\right)&t\cos\left(\frac{t}{\sqrt{2}}\right)\\
		t\cos\left(\frac{t}{\sqrt{2}}\right)&t\cos\left(\frac{t}{\sqrt{2}}\right)
	\end{pmatrix}.
\end{align*}
These Jacobi fields all vanish at $t=0$ and are linearly independent. But, there is no conjugate point that can be read of directly. So, we investigate whether there are linear combinations of Jacobi fields that define conjugate points. First, we observe that also the matrices defined by the Jacobi fields at some time $t$ are linearly independent as long as all $t$-dependent entries are non-zero. Hence, we are looking for linear combinations of Jacobi fields that vanish at some time $t$ where a $t$-dependent entry becomes zero.
All the Jacobi fields $J_k$ feature the terms $\sin\left(\frac{t}{\sqrt{2}}\right)$ or $\cos\left(\frac{t}{\sqrt{2}}\right)$
in some of their entry-functions.
The term $\sin\left(\frac{t}{\sqrt{2}}\right)$ becomes zero for multiples of $\sqrt{2}\pi$. The term $\cos\left(\frac{t}{\sqrt{2}}\right)$ becomes zero for $\sqrt{2}\left(\frac{\pi}{2} + k\pi\right)$, $k\in\mathbb{Z}$. The smallest positive root of the term $\frac12\left(t\cos\left(\frac{t}{\sqrt{2}}\right) + \sqrt{2}\sin\left(\frac{t}{\sqrt{2}}\right)\right)$ that features in $J_1$ is $t_1$, where $t_1 = \sqrt{2}t_\frac12^\mathrm{r}\approx 2.8690968494$. Recall that the geodesic length between the starting point of a geodesic and its first conjugate point is bounded below by the injectivity radius. Since the injectivity radius of the Stiefel manifold equipped with the canonical metric is bounded below by $\sqrt{\frac{4}{5}}\pi\approx 2.8099258924$, the smallest candidate to look for a linear combination of Jacobi fields vanishing at $t\geq\sqrt{\frac{4}{5}}\pi$ is $t_1$. Indeed, $J(t) \coloneq J_1(t) + \frac{t_1}{4\sqrt{2}}(J_3(t) - J_2(t))$ is a Jacobi field along $\gamma$ that vanishes at $t_1$. 
This results in $\gamma$ having its first conjugate point at $t_1 = \sqrt{2}t_\frac12^\mathrm{r}$.
\end{appendices}

\end{document}